\documentclass{amsart}

\usepackage{amssymb}
\usepackage{amsmath}
\usepackage{amsthm}
\usepackage{color}
\usepackage{hyperref,url}
\usepackage[shortlabels]{enumitem}
\usepackage{graphicx}
\usepackage{asymptote}
\usepackage[T1]{fontenc}

\newtheorem{theorem}{Theorem}[section]
\newtheorem{lemma}[theorem]{Lemma}
\newtheorem{corollary}[theorem]{Corollary}
\newtheorem{prop}[theorem]{Proposition}

\newtheorem{defn}[theorem]{Definition}

\newcommand{\Z}{\mathbb{Z}}

\newcommand{\R}{\mathbb{R}}

\newcommand{\T}{\mathbb{T}}

\newcommand{\e}{\epsilon}
\newcommand{\eps}{\varepsilon}

\newcommand{\lesim}{\lesssim}
\newcommand{\gesim}{\gtrsim}
\newcommand{\calC}{\mathcal{C}}
\newcommand{\calD}{\mathcal{D}}
\newcommand{\cD}{\mathcal{D}}
\newcommand{\calK}{\mathcal{K}}

\newcommand{\calL}{\mathcal{L}}
\newcommand{\calQ}{\mathcal{Q}}
\newcommand{\cQ}{\mathcal{Q}}
\newcommand{\calP}{\mathcal{P}}
\newcommand{\cP}{\mathcal{P}}
\newcommand{\calR}{\mathcal{R}}
\newcommand{\cR}{\mathcal{R}}
\newcommand{\calS}{\mathcal{S}}

\newcommand{\tA}{\tilde{A}}
\newcommand{\tB}{\tilde{B}}
\newcommand{\tC}{\tilde{C}}
\newcommand{\tX}{\tilde{X}}
\newcommand{\tY}{\tilde{Y}}

\newcommand{\intersect}{\sim}

\title[Incidence estimates for $\alpha$-dimensional tubes and $\beta$-dimensional balls]{Incidence estimates for $\alpha$-dimensional tubes and $\beta$-dimensional balls in $\R^2$}

\author{Yuqiu Fu}
\address{Department of Mathematics, MIT,
Cambridge, MA 02139}
\email{yuqiufu@mit.edu}

\author{Kevin Ren}
\address{Department of Mathematics, MIT,
Cambridge, MA 02139}
\email{kevinren@mit.edu}
\date{\today}

\begin{document}

\begin{abstract}
    We prove essentially sharp incidence estimates for a collection of $\delta$-tubes and $\delta$-balls in the plane, where the $\delta$-tubes satisfy an $\alpha$-dimensional spacing condition and the $\delta$-balls satisfy a $\beta$-dimensional spacing condition. Our approach combines a combinatorial argument for small $\alpha, \beta$ and a Fourier analytic argument for large $\alpha, \beta$. As an application, we prove a new lower bound for the size of a $(u,v)$-Furstenberg set when $v \ge 1, u + \frac{v}{2} \ge 1$, which is sharp when $u + v \ge 2$. We also show a new lower bound for the discretized sum-product problem.
\end{abstract}

\maketitle

\section{Introduction}
Let $0 < \delta \leq 1$ be a small parameter. We will work with $\delta$-tubes and $\delta$-balls in the plane $\R^2$. A {$\delta$-ball} is a ball of radius $\delta$. A {$\delta$-tube} is a $\delta \times 1$ rectangle. The {direction} of a rectangle is the vector pointing in the direction of its longest side. (This vector is only determined up to $\pm 1$.) 


\begin{defn}
    Let $P$ be a set of $\delta$-balls and $\T$ be a set of $\delta$-tubes. The number of incidences $I(P, \T)$ is the number of pairs $(p, t)$ of $\delta$-balls $p \in P$ and $\delta$-tubes $t \in \T$ such that $p$ intersects $t$: $p \cap t \neq \emptyset$.
\end{defn}
    
The basic problem we will consider is the following: Given a set of $\delta$-balls $P$ and a set of $\delta$-tubes $\T$ contained in the square $[0, 1]^2$, what is the maximum number of incidences $I(P, \T)$?


We will impose a spacing condition on the set of $\delta$-balls and the set of $\delta$-tubes. The spacing condition is standard, see e.g. \cite{hera-improved}.

\begin{defn}\label{def:ball_dim}
    For $0 \le \beta \le 2$ and $K \ge 1$, we call a set of $\delta$-balls $P$ contained in $[0, 1]^2$ a $(\delta, \beta, K)$-set of balls if for every $w \in [\delta, 1]$ and every ball $B_w$ of radius $w$,
    \begin{equation*}
        \# \{ p \in P \mid p \subset B_w \} \le K \cdot \left( \frac{w}{\delta} \right)^\beta.
    \end{equation*}
\end{defn}
In the definition of $(\delta, \beta, K)$-set, $K$ may depend on $\delta$. If $K$ is constant, then we drop $K$ from the notation. By taking $w = 3\delta$, any $\delta$-ball in a $(\delta, \beta, K)$-set of balls may intersect up to $\le 9K$ many other $\delta$-balls in the set.

We will impose an analogous condition on the set of $\delta$-tubes.
\begin{defn}\label{def:tubes_dim}
    For $0 \le \alpha \le 2$ and $K \ge 1$, we call a set of $\delta$-tubes $\T$ contained in $[0, 1]^2$ a $(\delta, \alpha, K)$-set of tubes if for every $w \in [\delta, 1]$ and every $w \times 2$ tube $T_w$,
    \begin{equation*}
        \# \{ t \in \T \mid t \subset T_w \} \le K \cdot \left( \frac{w}{\delta} \right)^\alpha.
    \end{equation*}
\end{defn}

\textit{Remark.} (1) For applications, one might take $K = C_\eps \delta^{-\eps}$ for some $\eps > 0$.

(2) Another common definition for $(\delta, \beta, C)$-sets of balls has the condition $\# \{ p \in P \mid p \subset B_w \} \le C \cdot |P| \cdot w^\beta$. This is a special case of Definition \ref{def:ball_dim} with $K = |P| \delta^\beta C$.

We can rephrase the problem as follows: given a $(\delta, \beta, K_\beta)$-set of balls $P$ and a $(\delta, \alpha, K_\alpha)$-set of tubes $\T$, what is the maximum number of incidences $I(P, \T)$?

In \cite{guth}, incidence problems for $\delta$-tubes with some spacing conditions were considered. They fix a parameter $1 \le W \le \delta^{-1}$ and choose $\T$ to be a collection of $W^2$ well-spaced $\delta$-tubes: each $W^{-1} \times 1$ rectangle in $\R^2$ contains at most one $\delta$-tube in $\T$. They also consider another spacing condition, where each $W^{-1} \times 1$ rectangle contains $\sim N_1$ many $\delta$-tubes in each direction, for a fixed $N_1$. Using a Fourier analytic approach, \cite{guth} proved sharp incidence estimates for well-spaced $\delta$-tubes. This Fourier analytic method is also used in \cite{gwz20,dgw20,gmw20,FGMdirichlet, FGMsharplevel} to derive incidence estimates, decoupling estimates, and square function estimates. 

Regarding our question, we will prove the following main theorem:
\begin{theorem}\label{main}
    Suppose $\alpha, \beta$ satisfy $0 \le \alpha, \beta \le 2$, and let $K_\alpha, K_\beta \ge 1$. For every $\eps > 0$, there exists $C = C_\eps K_\alpha K_\beta$ with the following property: for every $(\delta, \beta, K_\beta)$-set of balls $P$ and $(\delta, \alpha, K_\alpha)$-set of tubes $\T$ contained in $[0, 1]^2$, the following bound holds:
    \begin{equation*}
        I(P, \T) \le C \cdot \delta^{-f(\alpha, \beta) - \eps},
    \end{equation*}
    where $f(\alpha, \beta)$ is defined as in Figure \ref{fig:answers}. These bounds are sharp up to $C \cdot \delta^{-\eps}$.
\end{theorem}
\begin{figure}[h]
    \includegraphics[scale=0.7]{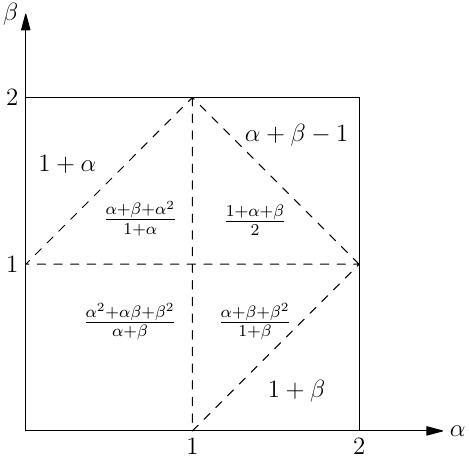}
    \caption{Values of $f(\alpha,\beta)$}
    \label{fig:answers}
\end{figure}

For $\alpha, \beta > 1$, we have the following refined result:
\begin{theorem}\label{mainbigbig}
Fix $\eps > 0$, and let $c^{-1} = \max(\alpha + \beta - 1, 2)$. There exists $C_\eps > 0$ such that the following holds: for any $(\delta, \beta, K_\beta)$-set of balls $P$ and $(\delta, \alpha, K_\alpha)$-set of tubes $\T$ contained in $[0, 1]^2$, we have the following incidence bound:
\begin{equation*}
    I(P, \T) \le C_\eps \delta^{-c-\eps} (K_\alpha K_\beta)^c |P|^{1-c} |\T|^{1-c}.
\end{equation*}
\end{theorem}
As an application of Theorem \ref{mainbigbig}, we prove a lower bound for the minimal Hausdorff dimension of a $(u, v)$-Furstenberg set. There has been much study of $(u, v)$-Furstenberg sets; what follows is an abbreviated exposition borrowing from \cite{furstenberg}. A set $A \subset \R^2$ is called a $(u, v)$-Furstenberg set if there exists a family of lines $\calL$ with $\dim_H \calL = v$ and $\dim_H (A \cap \ell) \ge u$ for all $\ell \in \calL$, where $\dim_H K$ denotes the Hausdorff dimension of $K$. The $(u, v)$-Furstenberg set problem asks for bounds on $\gamma(u, v) := \inf \{ \dim_H (A) : A \text{ is an } (u, v)\text{-Furstenberg set} \}$. The case $v = 1$ has attracted considerable interest. While it is conjectured that $\gamma(u, 1) = \frac{1 + 3u}{2}$ for $u>0$, the best known bounds are from Wolff \cite{wolff}, Bourgain \cite{bourgain2003erdHos}, and Orponen and Shmerkin \cite{orp-shm}, whose work shows $\gamma(u, 1) \ge \max(u + \frac{1}{2}, 2u + \eps(u))$ for some small constant $\eps(u) > 0$ when $u<1$ and $\eps(u) = 0$ when $u = 1.$

For more general $v \in [0, 2]$, work of Molter and Rela \cite{molt}, H\'era \cite{hera}, H\'era, M\'athe, and K\'eleti \cite{hera-mathe-keleti}, Lutz and Stull \cite{lutz}, H\'era, Shmerkin, and Yavicoli \cite{hera-improved}, Orponen and Shmerkin \cite{orp-shm}, and Shmerkin and Wang \cite{shmerkin2022dimensions} show that
\begin{equation}\label{previousF}
    \gamma(u, v) \ge \begin{cases}
        u + v, & \text{ if } v \le u, u\in (0,1] \\
        2u + \eps(u, v), & \text{ if } u < v \le 2u , u\in (0,1], \\
        u + \frac{v}{2} + \eps(u, v), & \text{ if } 2u < v \leq 2, u\in (0,1]. 
    \end{cases}
\end{equation}
Recently, Dabrowski, Orponen, and Villa in \cite{furstenberg} showed for $u > \frac{1}{2}, v > 1$ that $\gamma(u, v) \ge 2u + (v-1)(1-u)$. Our result improves on \eqref{previousF} and \cite{furstenberg} for $(u, v)$ pairs satisfying $v > 1 + \eps(u, v)$ and $u + \frac{v}{2} \ge 1 + \eps(u, v)$. 

\begin{theorem}\label{thm:furstenberg}
    For $1 \le v \le 2$ and $0 < u \le 1$, a $(u, v)$-Furstenberg set $A$ has $\dim_H A \ge \min(2u + v - 1, u + 1)$. This result is sharp when $u + v \ge 2$.
\end{theorem}
Note that the bound $\dim_H A \ge 2u + v - 1$ was proved in \cite{molt} for $0 < u, v \le 1$.

As a quick corollary of Theorem \ref{thm:furstenberg}, we obtain the following variant of Marstrand's slicing theorem \cite{marstrand1954some}, which states that for all directions $\theta \in S^1$, then for a.e. line $\ell$ in direction $\theta$, we have $\dim_H (A \cap \ell) \le t-1$. In fact, we are able to bound the dimension of the exceptional set of lines for which $\dim_H (A \cap \ell) > t-1$.

\begin{corollary}\label{cor:marstrand}
    Let $A \subset \R^2$ be a set with $\dim_H (A) = t > 1$, and let $\calL$ be a set of lines such that $\dim_H (A \cap \ell) > t-1$ for all $L \in \calL$. Then $\dim_H (\calL) \le 3-t$.
\end{corollary}

\begin{proof}
    Let $\calL_u$ be the set of lines $\ell$ such that $\dim_H (A \cap \ell) = u$. Suppose that $\dim_H \calL_u \ge 3-t$ for some $u > t-1$. Then $A$ is a $(u, 3-t)$-Furstenberg set with $u + 3 - t \ge 2$. Hence, by Theorem \ref{thm:furstenberg}, we get $\dim_H (A) \ge u+1 > t$, contradiction to $\dim_H (A) = t$. Thus, we actually have $\dim_H (\calL_u) < 3-t$ for all $u > t-1$. Let $\{ u_i \}_{i=1}^\infty$ be a sequence converging to $t-1$ from above. Then $\calL = \cup_{i \ge 1} \calL_{u_i}$, so $\dim_H (\calL) = \sup_{i \ge 1} \dim_H (\calL_{u_i}) \le 3-t$, as desired.
\end{proof}

Our approach also allows us to obtain the following discretized sum-product estimate:
\begin{corollary}\label{cor:sum-product}
    Let $0 < \delta \le 1$, $u, v, v' \in [0, 1]$ with $v + v' > 1$, and $K_u, K_v, K_{v'} \ge 1$. Let $A, B, C \subset [1, 2]$ be sets of disjoint $\delta$-balls such that $A$ is a $(\delta, u, K_u)$-set, $B$ is a $(\delta, v, K_v)$ set, and $C$ is a $(\delta, v', K_{v'})$ set. For a set $E \subset \R$, let $|E|_\delta$ denote the minimum number of $\delta$-balls needed to cover $E$. Then for $c = \max(u + v + v', 2)^{-1}$,
    \begin{equation*}
        \max(|A+B|_\delta, |A \cdot C|_\delta) \gesim K_u^{-\frac{c}{2(1-c)}} K_v^{-\frac{c}{2(1-c)}} \delta^{\frac{c}{2(1-c)} + \eps} |B|^{\frac{c}{2(1-c)}} |C|^{\frac{c}{2(1-c)}} |A|^{\frac{1}{2(1-c)}}.
    \end{equation*}
\end{corollary}
This corollary strengthens Corollary 1.11 of \cite{furstenberg} when $A$ is a $(\delta, s, \delta^{-\eps})$-set with $|A| \sim \delta^{-s}$. If we apply Corollary \ref{cor:sum-product} for $A = B = C$, $K_u = \delta^{-\eps}$, $|A| \sim \delta^{-s}$, and $s \in (\frac{1}{2}, 1)$, we get the non-trivial sum-product estimate
\begin{equation*}
    \max(|A+A|_\delta, |A \cdot A|_\delta) \gesim \begin{cases}
        \delta^{-2s + \frac{1}{2} + 2\eps}, & s < \frac{2}{3}, \\
        \delta^{-\frac{s+1}{2} + 2\eps}, & s > \frac{2}{3}.
    \end{cases}
\end{equation*}
This improves on results of Chen \cite{chen2020discretized} for every $s \in (\frac{1}{2}, 1)$ and Guth, Katz, and Zahl \cite{guth2021discretized} for $1 > s > (\sqrt{1169} - 21)/26 \approx 0.5073$.

Finally, we remark that Theorem \ref{mainbigbig} and Theorem \ref{thm:furstenberg} can be generalized to the case of $\delta$-balls and $\delta$-flats (i.e. $\delta$-neighborhoods of $(n-1)$-planes) in $\R^n$; we will explore this generalization in a subsequent paper.

We conclude this introductory section by describing the organization of the paper.
In Section 2, we will show the estimates in Theorem \ref{main} are sharp up to $C_\e \delta^{-\eps}$, by constructing suitable examples.

We will then prove the upper bound of Theorem \ref{main} by analyzing different cases for $\alpha, \beta$. In Section 3, we will use a combinatorial argument (the $L^2$ argument as in \cite{cordoba}) to resolve the case where $\alpha \le 1$ or $\beta \le 1$. In Section 4, we will induct on scale $\delta$ to prove Theorem \ref{mainbigbig}, the case where $\alpha, \beta \ge 1$. The starting point of this argument will be the Fourier-analytic Proposition 2.1 from \cite{guth}, which was inspired by ideas of Orponen \cite{orp} and Vinh \cite{vinh}. Finally, we derive Theorem \ref{main}, Theorem \ref{thm:furstenberg}, and Corollary \ref{cor:sum-product} in Section 5.

\vspace{4mm}
\noindent
\textbf{Notation.} We will use $A \gesim B$ to represent $A \ge CB$ for a constant $C$, and $A \lesim B$ to represent $A \le CB$. The constant $C$ is independent of the scale $\delta$ and the dimension parameters $\alpha, \beta, K_\alpha, K_\beta$. We will use $A \sim B$ to represent $A \gesim B$ and $A \lesim B$. Finally, we let $A \gesim_\eps B$ to denote $A \ge CB$ for a constant $C$ which depends on $\eps$, and define $A \lesim_\eps B$, $A \sim_\eps B$ similarly.

For a finite set $A$, typically a set of $\delta$-tubes or $\delta$-balls, let $\# A$ or $|A|$ denote its cardinality. For a subset $A \subset \R^2$, let $|A|_\delta$ denote the least number of $\delta$-balls needed to cover $A$.

For a set $P$ of $\delta$-balls and a subset $A \subset \R^2$, let $P \cap A := \{ p \in P \mid p \subset A \}$.

The angle between two $\delta$-tubes $s$ and $t$, or $\angle(s, t)$, is the acute angle between their directions.

For two sets $A$ and $B$ in $\R^2$, we say $A$ and $B$ intersect if $A \cap B \neq \emptyset$.

For a $\delta$-ball $p$ and $S \ge 1$, define the $S$-thickening $p^S$ to be the $S\delta$-ball concentric with $p$. For a $\delta$-tube $t$, let $t^S$ denote the $S\delta$-tube coaxial with $t$. Finally, for a set of $\delta$-balls $P$ (respectively set of $\delta$-tubes $\T$), let $P^S := \{ p^S : p \in P \}$ (respectively $\T^S := \{ t^S : t \in \T \}$).

We say two $\delta$-tubes $s, t$ are \textit{essentially identical} if they intersect and their angle is $\le \delta$. Otherwise, they are essentially distinct, and we say a collection $\T$ of $\delta$-tubes is essentially distinct if the tubes in $\T$ are pairwise essentially distinct.

\vspace{4mm}
\noindent
\textbf{Acknowledgements.}
We wish to thank the MIT SPUR program, funded by the MIT Department of Mathematics, where most of this research was conducted.
We would like to thank Larry Guth for suggesting this problem and helpful discussions, Ankur Moitra and David Jerison for helpful discussions, and Slava Gerovitch for running the SPUR program. We thank Damian D{\k{a}}browski for suggesting the statement and proof of Corollary \ref{cor:marstrand}. Finally, we thank the anonymous referee for helpful suggestions on the exposition and insightful comments which simplified the proof of Theorem \ref{mainbigbig} and led us to formulate Theorem \ref{thm:furstenberg} and Corollary \ref{cor:sum-product}.

\section{Constructions}
We start with the sharpness part of Theorem \ref{main}. We will construct $(\delta, \alpha)$-sets of tubes and $(\delta, \beta)$-sets of balls such that the number of incidences is at least $\delta^{-f(\alpha, \beta)}$, where $f$ was defined as in Theorem \ref{main}. 
We divide the constructions into four cases. Construction 1 is the main construction that works for most $\alpha$ and $\beta$. Constructions 2, 3, 4 can be considered as auxiliary constructions which take care of exceptional values of $\alpha, \beta$ not covered in Construction 1. The constructions will all take place inside a $1 \times 1$ square. In the constructions, some of the $\delta$-tubes may not be fully contained within the $1 \times 1$ square, but we will ignore this minor detail. For ease of notation, let $D = \delta^{-1}$.

\subsection{Construction 1}
In this construction, we assume $\alpha < \beta + 1, \beta < \alpha + 1$, and  $\alpha + \beta < 3$. Let $a = \min(\alpha, 1), b = \min(\beta, 1)$. Our goal is to construct a $(\delta,\beta)$-set of $\delta$-balls and a $(\delta,\alpha)$-set of $\delta$-tubes with at least $\delta^{-\frac{a\alpha + b\beta + ab}{a+b}}$ incidences.

\begin{figure}[h]
    \begin{center}
        \includegraphics[scale=0.8]{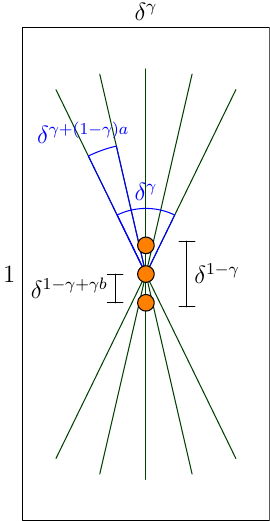}
        \hspace{0.5cm}
        \includegraphics[scale=0.8]{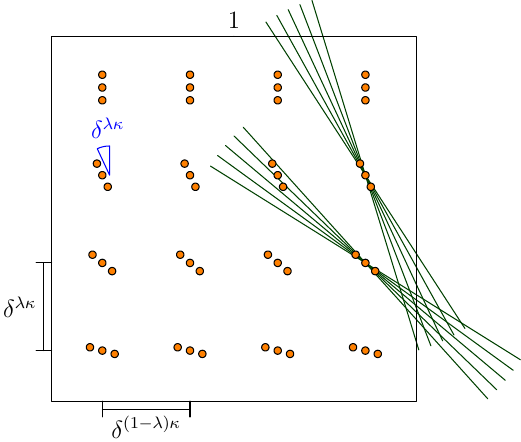}
    \end{center}
    \caption{Construction 1.}
    \label{fig:case1}
\end{figure}

To describe the construction, we will need a few auxiliary variables. Recall $a = \min(\alpha, 1), b = \min(\beta, 1)$, and we will eventually choose $\gamma, \kappa, \lambda$ as parameters in $[0, 1]$.
Refer to Figure \ref{fig:case1}. The left picture depicts a single bundle with $D^{(1-\gamma)a}$ many $\delta$-tubes and $\sim D^{\gamma b}$ many $\delta$-balls. The $\delta$-tubes are rotates of a single central $\delta$-tube $t$, and the angle spacing between $\delta$-tubes is $\delta^{\gamma + (1-\gamma)a}$, so that the maximal angle of two $\delta$-tubes in the bundle is $\delta^\gamma$. By trigonometry, the intersection of all the tubes contains a $\delta \times (\sim \delta^{1-\gamma})$ rectangle with the same center and direction as the central $\delta$-tube $t$. We may thus place $\sim D^{\gamma b}$ many $\delta$-balls in the rectangle, spaced a distance of $\delta^{1-\gamma+\gamma b}$ apart; then each ball of the bundle will intersect each tube in the bundle. Furthermore, since the maximum angle between two $\delta$-tubes in the bundle is $\delta^\gamma$, we see that the  bundle fits inside a $\delta^\gamma \times 1$ rectangle. 

It might be helpful to observe that the configuration of $\delta$-balls is ``dual'' to the configuration of $\delta$-tubes in a bundle, in the sense that the $\delta$-balls in a bundle are evenly spaced along the central axis, while the $\delta$-tubes are evenly spaced in direction.

    

In the right picture, there are $D^\kappa$ bundles in $[0, 1]^2$. The bundles are arranged in a $D^{(1-\lambda)\kappa} \times D^{\lambda\kappa}$ grid, with the horizontal spacing $\delta^{(1-\lambda)\kappa}$ and the vertical spacing $\delta^{\lambda\kappa}$. The bundles in the same row are translates of each other; two adjacent bundles in the same column are $\delta^{\lambda \kappa}$ rotates of each other. 

If $\T$ is the set of $\delta$-tubes and $P$ is the set of $\delta$-balls in the configuration, then we see that $|\T| \sim D^{(1-\gamma)a + \kappa}$ and $|P| \sim D^{\gamma b + \kappa}$.

Intuitively, we can regard $\lambda$ as controlling the ``aspect ratio'' of the bundle configuration. If $\lambda = 1$, then all the bundles are rotated copies of each other, arranged vertically; if $\lambda = 0$, then all the bundles are translated copies of each other, arranged horizontally. For the right values of $\lambda$, our constructed $\T$ will be a $(\delta, \alpha)$-set of $\delta$-tubes and $P$ will be a $(\delta, \beta)$-set of $\delta$-balls.

Now, we choose suitable values for our parameters $\gamma, \kappa, \lambda$. We first choose $\gamma = \frac{a - \alpha + \beta}{a + b}$ and $\kappa = \alpha - (1-\gamma)a = \frac{a \beta + b \alpha - ab}{a + b}$. Then, we apply the following Lemma to choose $\lambda$ and also check that $\gamma, \kappa \in [0, 1]$.
\begin{lemma}\label{app2}
(a) We have $0 \le \gamma, \kappa \le 1$.

(b) There exists $0 \le \lambda \le 1$ such that $\T$ is a $(\delta, \alpha)$-set of $\delta$-tubes and $P$ is a $(\delta, \beta)$-set of $\delta$-balls.
\end{lemma}
The proof is computational, and we defer it to the Appendix. Now with this choice of parameters, we find $|\T| \sim D^{(1-\gamma)a + \kappa} = D^\alpha$ and $|P| \sim D^{\gamma b + \kappa} = D^\beta$.
Finally, since $|\T| \sim D^\alpha$, and each $\delta$-tube $t \in \T$ intersects $\sim D^{\gamma b}$ many $\delta$-balls of the bundle of $t$, we get $I(P, \T) \gesim D^{\alpha} D^{\gamma b} = D^{\frac{a\alpha + b\beta + ab}{a+b}}$.
\begin{figure}[h]
    \includegraphics[scale=0.8]{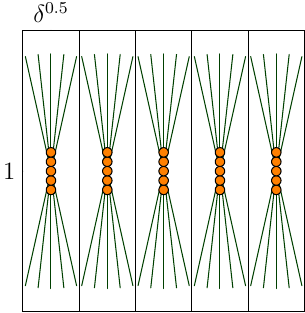}
        
        
        
        
        
    \caption{The special case $\alpha = \beta = 1$.}
    \label{fig:special}
\end{figure}

The prototypical example is $\alpha = 1, \beta = 1$, in which $\gamma = \kappa = 0.5$. In this case, the possible values for $\lambda$ are $0 \le \lambda \le \frac{1}{2}$. If we choose $\lambda = 0$, then we get a series of $D^{0.5}$ horizontally spaced, parallel bundles, as in Figure \ref{fig:special}.


\subsection{Construction 2}
For this construction, we will assume $\alpha \ge \beta + 1$. Our goal is to obtain $\delta^{-(\beta + 1)}$ incidences.

\begin{figure}
\begin{center}
\includegraphics[scale=0.8]{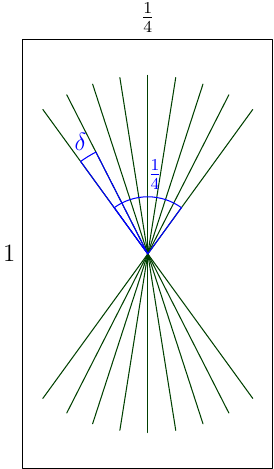}
\hspace{0.2cm}
\includegraphics[scale=0.8]{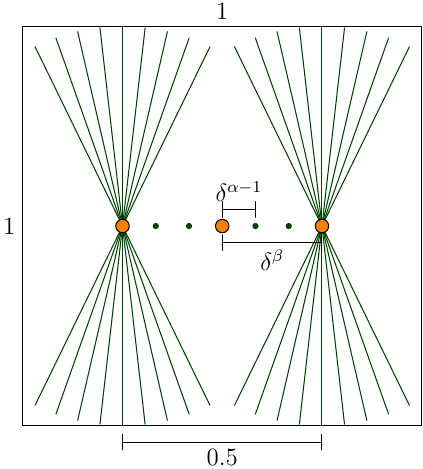}
\end{center}
\caption{Construction 2.}
\label{fig:case2}
\end{figure}

Refer to Figure \ref{fig:case2}. In each bundle, there are $\sim D$ many $\delta$-tubes, each separated by angle $\delta$. 
Thus, we can fit the bundle inside a $\frac{1}{4} \times 1$ rectangle. We arrange $\sim D^{\alpha-1}$ bundles as in the right figure, separated by distance $\sim \delta^{\alpha-1}$, such that the centers of the bundles lie within a segment of length $\frac{1}{2}$ centered at the unit square's center. Then, we place $D^\beta$ many $\delta$-balls at some of the centers of the bundles, such that the $\delta$-balls are $\delta^\beta$-separated. Thus, there are $D^\alpha$ many $\delta$-tubes and $D^\beta$ many $\delta$-balls in the configuration.

Let $\T$ be the set of $\delta$-tubes and $P$ be the set of $\delta$-balls. We will show that $\T$ is a $(\delta, \alpha)$-set of tubes and $P$ is a $(\delta, \beta)$-set of balls.

Fix $w \in [\delta, 1]$ and a $w \times 2$ rectangle $R_w$; we will count how many $\delta$-tubes in $\T$ are in $R_w$. There are two main contributions.

\begin{itemize}
    \item $R_w$ can contain tubes from $\lesim \lceil \frac{w}{\delta^{\alpha - 1}} \rceil$ bundles of $|\T|$.
    
    \item For each bundle, $R_w$ can contain $\lesim \frac{w}{\delta}$ $\delta$-tubes.
\end{itemize}

Thus, $R_w$ contains at most $N$ $\delta$-tubes in $\T$, where (using $w \in [\delta, 1]$ and $\alpha \ge 1$):
\begin{equation*}
    N \lesim \left( \frac{w}{\delta^{\alpha - 1}} + 1 \right) \cdot \frac{w}{\delta} = \frac{w^2}{\delta^\alpha} + \frac{w}{\delta} \le 2 \cdot \left( \frac{w}{\delta} \right)^\alpha.
\end{equation*}
This means $\T$ is a $(\delta, \alpha)$-set of tubes.

Now, we verify that $P$ is a $(\delta, \beta)$-set of balls. Fix $w \in [\delta, 1]$ and a ball $B_w$ of radius $w$; we will count how many $\delta$-balls in $P$ are in $B_w$. Note that $B_w$ can intersect at most $N$ many $\delta$-balls, where (using $\beta \le \alpha - 1 \le 1$):
\begin{equation*}
    N \lesim \lceil \frac{w}{\delta^\beta} \rceil \le \frac{w}{\delta^\beta} + 1 \le 2 \left( \frac{w}{\delta} \right)^\beta.
\end{equation*}

Thus, $P$ is a $(\delta, \beta)$-set. Finally, each $\delta$-ball in $P$ intersects $\sim D$ many $\delta$-tubes of $\T$, so $I(P, \T) \sim D^\beta \cdot D = D^{\beta + 1}$.

\subsection{Construction 3}
For this construction, we will assume $\beta \ge \alpha + 1$. Our goal is to obtain $\delta^{-(\alpha + 1)}$ incidences.

\begin{figure}[h]
\includegraphics[scale=0.8]{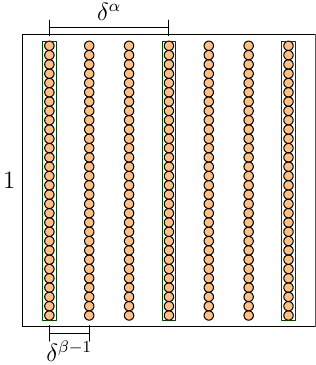}
    
    
    
    
    
    
\caption{Construction 3.}
\label{fig:case3}
\end{figure}

Refer to Figure \ref{fig:case3}. There are $D^{\beta-1}$ columns of $D$ many $\delta$-balls each. On $D^\alpha$ of the columns, there is a $\delta$-tube. The $\delta$-tube-containing columns are separated by distance $\delta^\alpha$. Thus, there are $D^\beta$ $\delta$-balls and $D^\alpha$ $\delta$-tubes. Note that Construction 3 is ``dual'' to Construction 2, in the sense that a bundle of direction-separated $\delta$-tubes is replaced by a bundle of evenly-spaced $\delta$-balls.

The $\delta$-tubes are a $(\delta, \alpha)$-set of tubes and the $\delta$-balls are a $(\delta, \beta)$-set of balls by a similar argument to Construction 2. Finally, each $\delta$-tube contains $D$ many $\delta$-balls, so $I(P, \T) = D^\alpha \cdot D = D^{\alpha + 1}$.

\subsection{Construction 4}
For this construction, we will assume $\alpha + \beta \ge 3$. Our goal is to obtain $\delta^{-(\alpha + \beta - 1)}$ incidences.

\begin{figure}[h]
\includegraphics[scale=0.8]{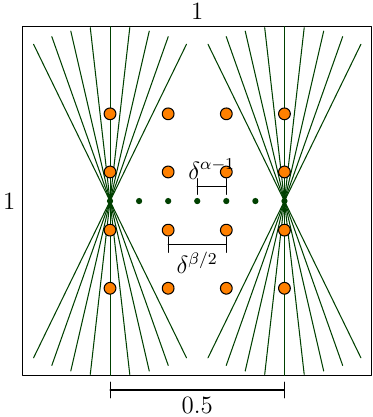}
    
    
    
    
    
    
    
\caption{Construction 4.}
\label{fig:case4}
\end{figure}

Refer to Figure \ref{fig:case4}. The bundles of $\delta$-tubes are the same as Construction 2. We then arrange $\sim D^\beta$ many $\delta$-balls in a $D^{\beta/2} \times D^{\beta/2}$ grid, such that adjacent $\delta$-balls are separated by distance $\sim \delta^{\beta/2}$. We confine the $\delta$-balls to a $\frac{1}{2} \times \frac{1}{2}$ square $\calS$ concentric with the large $1 \times 1$ square. Let $\T$ be the set of $\delta$-tubes and $P$ be the set of $\delta$-balls in this configuration.

From Construction 2, $\T$ is a $(\delta, \alpha)$-set of tubes. We now show that $P$ is a $(\delta, \beta)$-set of balls.

Fix $w \in [\delta, 1]$ and a ball $B_w$ of radius $w$; we will count how many $\delta$-balls in $P$ are in $B_w$. Note that $B_w$ can intersect at most $N$ many $\delta$-balls in $P$, where
\begin{equation*}
    N \lesim \left( \lceil \frac{w}{\delta^{\beta/2}} \rceil \right)^2 \le \left( \frac{w}{\delta^{\beta/2}} + 1 \right)^2 \le 2 \left( \frac{w^2}{\delta^\beta} + 1 \right) \le 4\left( \frac{w}{\delta} \right)^\beta.
\end{equation*}
Also, the $\delta$-balls in $P$ are essentially distinct, so $P$ is a $(\delta, \beta)$-set of balls.

Finally, we will count the number of incidences. For a bundle centered at some point $O \in \calS$, the $\delta$-tubes in the bundle cover a double cone with apex $O$ and angle $\frac{1}{4}$. This double cone intersects square $\calS$ in a polygonal region with positive area, so it contains a positive fraction of the balls in $P$. Hence, the number of incidences between a given bundle and $P$ is $\gesim D^\beta$. There are $D^{\alpha-1}$ bundles in $\T$, so $I(P, \T) \gesim D^\beta D^{\alpha-1} = D^{\alpha + \beta - 1}$.

\section{Combinatorial upper bound}
We will first prove the upper bound for $\alpha \le 1$ or $\beta \le 1$. We further casework on whether $\alpha < \beta$ or $\alpha \ge \beta$, which are handled by Theorems \ref{ballsdom} and \ref{tubesdom} below.
\begin{theorem}\label{ballsdom}
Let $P$ be a $(\delta, \beta, K_\beta)$-set of balls and $\T$ be a $(\delta, \alpha, K_\alpha)$-set of tubes. Let $D = \delta^{-1}$. Let $b = \min(\beta, 1)$, and assume $b \ge \alpha$. Then for any $\eps > 0$, there exists $C_\eps > 0$ such that
\begin{equation*}
    I(P, \T)^{\alpha + b} \le C_\eps D^{\alpha b (1+\eps)} K_\beta^\alpha K_\alpha^b |P|^b |\T|^\alpha.
\end{equation*}
\end{theorem}
\begin{theorem}\label{tubesdom}
Let $P$ be a $(\delta, \beta, K_\beta)$-set of balls and $\T$ be a $(\delta, \alpha, K_\alpha)$-set of tubes. Let $D = \delta^{-1}$. Let $a = \min(\alpha, 1)$ and assume $a \ge \beta$. Then for any $\eps > 0$, there exists $C_\eps > 0$ such that
\begin{equation*}
    I(P, \T)^{a + \beta} \le C_\eps D^{a \beta (1+\eps)} K_\beta^a K_\alpha^\beta |P|^\beta |\T|^a.
\end{equation*}
\end{theorem}

\begin{proof} We will first prove Theorem \ref{ballsdom}. Then to prove Theorem \ref{tubesdom}, it suffices to prove Theorem \ref{ballsdom} and apply duality. For more details on duality, see Section 6.1 of \cite{furstenberg}. Hence, we will concentrate on proving Theorem \ref{ballsdom}.

\textbf{Notation.} For a $\delta$-tube $t$ and $\delta$-ball $p$, we use $p \intersect t$ to denote $p \cap t \neq \emptyset$. 

If $\alpha = 0$ then the result follows from the trivial bound $I(P, \T) \le |P| |\T| \lesim |P| \cdot K_\alpha$, so we assume $\alpha > 0$. Recall that $I(P, \T)$ is the number of pairs $(t, p) \in \T \times P$ such that $p \intersect t$. Let $\T(p) = \{ t \in \T \mid p \intersect t \}$. Define
\begin{equation*}
    J(P, \T) = \sum_{p \in P} |\T(p)|^{(b + \alpha)/\alpha}.
\end{equation*}
We first relate $J$ to $I$.  By H\"{o}lder's inequality, we have
\begin{equation}\label{iandj}
    J^{\alpha} \ge \frac{1}{|P|^b} \left( \sum_{p \in P} |\T(p)| \right)^{b + \alpha} = \frac{I(P, \T)^{b + \alpha}}{|P|^b}.
\end{equation}
Next, we estimate $J(P, \T)$. For a given $\delta$-tube $t \in \T$, let
\begin{equation*}
    j(t) = \sum_{p \intersect t} |\T(p)|^{b/\alpha}.
\end{equation*}
Then, we have
\begin{multline}\label{Jandj}
    J(P, \T) = \sum_{p \in P} |\T(p)| \cdot |\T(p)|^{b/\alpha} = \sum_{p \in P} \sum_{t \sim p} |\T(p)|^{b/\alpha} \\ = \sum_{t \in \T} \sum_{p \sim t} |\T(p)|^{b/\alpha} = \sum_{t \in \T} j(t).
\end{multline}
The main claim is the following:
\begin{lemma}\label{mainclaim}
There exists $C_\eps > 0$ such that $j(t) \lesim C_\eps K_\alpha^{b/\alpha} K_\beta D^{b + \eps/\alpha}$ for any $t \in \T$.
\end{lemma}
To prove Lemma \ref{mainclaim}, we introduce some notation. Let $\T_\delta (t) = \{ s \in \T \mid s \cap t \neq \emptyset, \angle(s, t) \le 2\delta \}$ and for $w \ge 2\delta$,
\begin{equation*}
    \T_w (t) = \{ s \in \T \mid s \cap t \neq \emptyset, w \le \angle(s, t) \le 2w \}.
\end{equation*}
We will now prove two lemmas involving $\T_w (t)$.
\begin{lemma}\label{intersect2}
$|\T_w (t)| \lesim K_\alpha \left( \frac{w}{\delta} \right)^\alpha$.
\end{lemma}

\begin{proof}
Let $R$ be the $200w \times 2$ rectangle with the same center as $t$ such that the length-2 side of $R$ is parallel to the length-1 side of $t$.
By trigonometry, we observe that any $\delta$-tube $s \subset [0, 1]^2$ with $s \cap t \neq \emptyset$ and $\angle(s, t) \le 2w$ must be contained in $R$. Since $\T$ is a $(\delta, \alpha, K_\alpha)$-set of tubes, there are at most $K_\alpha \cdot \left( \frac{200w}{\delta} \right)^\alpha$ tubes of $\T$ contained in $R$. Thus, since $200^\alpha \le 200^2 \lesim 1$, we obtain the desired bound $|\T_w (t)| \lesim K_\alpha \cdot \left( \frac{w}{\delta} \right)^\alpha$.
\end{proof}
\begin{lemma}\label{doublecount}
Fix $t \in \T$. For any $\delta < w < \frac{\pi}{2}$, we have
    \begin{equation*}
        \sum_{p \intersect t} |\T(p) \cap \T_w (t)| \lesim |\T_w (t)| \cdot \frac{K_\beta}{w^b}.
    \end{equation*}
\end{lemma}

\begin{proof}
We use a double counting argument. The left-hand side counts the number of pairs $(p, s) \times P \times \T_w (t)$ with $p \intersect s$ and $p \intersect t$. For each $s \in \T_w (t)$, $s \cap t$ is contained in a $\delta \times \frac{\delta}{w}$ rectangle $R_w$. To upper-bound the number of $\delta$-balls of $P$ in $R_w$, we split into cases:
\begin{itemize}
    \item If $\beta \ge 1$, then cover $R_w$ with $\sim \frac{1}{w}$ many $10\delta$-balls $q_i$, such that any $\delta$-ball that intersects $R_w$ must lie in some $q_i$. By dimension, $q_i$ contains at most $\lesim K_\beta \cdot 10^\beta \lesim K_\beta$ many $\delta$-balls of $P$, so $R_w$ intersects $\lesim \frac{1}{w} K_\beta$ many $\delta$-balls of $P$.
    
    \item If $\beta < 1$, then $R_w$ is contained in a ball of radius $\frac{\delta}{w}$, so since $P$ is a $(\delta, \beta, K_\beta)$-set, we see that $R_w$ intersects $\lesim K_\beta \left( \frac{1}{w} \right)^\beta$ many $\delta$-balls of $P$.
\end{itemize}
Thus, for each $s \in \T_w (t)$, there are at most $\lesim K_\beta \left( \frac{1}{w} \right)^b$ many $\delta$-balls $p \in P$ with $p \sim s$ and $p \sim t$, which proves the Lemma.
\end{proof}

\begin{proof}[Proof of Lemma \ref{mainclaim}]
Now let $\calK = \{ \delta, 2\delta, 4\delta, \dots, \delta \cdot 2^{\lceil \log_2 \delta^{-1} \rceil} \}$, $C_1 = |\calK|^{b/\alpha - 1}$, and $C = C_1 |\calK|$. Since $|\calK| = \log D$, we have $C \lesim_\eps D^{\eps/\alpha}$. Now, we perform the following calculation: 
\begin{align*}
    j(t) &= \sum_{p \intersect t} \left( \sum_{w \in \calK} |\T(p) \cap \T_w (t)| \right)^{b/\alpha} \\
        &\le \sum_{p \intersect t} C_1 \sum_{w \in \calK} |\T(p) \cap \T_w (t)|^{b/\alpha} \\
        &\le C_1 \sum_{w \in \calK} \sum_{p \intersect t} |\T(p) \cap \T_w (t)| \cdot |\T_w (t)|^{b/\alpha - 1} \\
        &\lesim C_1 \sum_{w \in \calK} |\T_w (t)|^{b/\alpha} \cdot \frac{K_\beta}{w^b} \\
        &\lesim C \cdot K_\alpha^{b/\alpha} K_\beta \cdot D^b \lesim_\eps K_\alpha^{b/\alpha} K_\beta \cdot D^{b+\eps/\alpha}.
\end{align*}
The first line uses the fact that the sets $\T_w (t)$ cover $\T(p)$. The second line follows from H\"{o}lder's inequality. The third line follows from $|\T(p) \cap \T_w (t)| \le |\T_w (t)|$ and $\frac{b}{\alpha} \ge 1$. The fourth line follows from Lemma \ref{doublecount}. The fifth line follows from Lemma \ref{intersect2}.
\end{proof}

Using Equations \eqref{iandj}, \eqref{Jandj} and Lemma \ref{mainclaim}, we can finish the proof of Theorem \ref{ballsdom}.
\begin{equation*}
    I(P, \T)^{b + \alpha} \le |P|^b J^\alpha \lesim_\eps |P|^b \left( \sum_{t \in \T} K_\alpha^{b/\alpha} K_\beta \cdot D^{b+\eps/\alpha} \right)^\alpha = K_\beta^\alpha K_\alpha^b |P|^b |\T|^\alpha D^{b \alpha + \eps}.
\end{equation*}
This is the desired result.
\end{proof}

\section{Fourier analytic upper bound}
We will now prove an upper bound for $I(P, \T)$ when $P$ is a $(\delta, \alpha, K_\alpha)$-set and $\T$ is a $(\delta, \beta, K_\beta)$-set of tubes, for $\alpha \ge 1$ and $\beta \ge 1$. The proof method is using the high-low method in Fourier analysis.

\subsection{A Fourier Analytic result}


    
We will need a variant of Proposition 2.1 from \cite{guth}. The version presented here is a modest refinement of \cite[Proposition 2.1] {bradshaw2020incidence}. First, we review some notation. We say two $\delta$-tubes $s, t$ are \textit{essentially identical} if they intersect and their angle is $\le \delta$. Otherwise, they are essentially distinct, and we say a collection $\T$ of $\delta$-tubes is essentially distinct if the tubes in $\T$ are pairwise essentially distinct.

For a $\delta$-ball $p$ and $S \ge 1$, define the $S$-thickening $p^S$ to be the $S\delta$-ball concentric with $p$. For a $\delta$-tube $t$, let $t^S$ denote the $S\delta$-tube coaxial with $t$. Finally, for a set of $\delta$-balls $P$ (respectively set of $\delta$-tubes $\T$), let $P^S := \{ p^S : p \in P \}$ (respectively $\T^S := \{ t^S : t \in \T \}$).
\begin{prop}\label{twopointone}
Fix a small $\eps > 0$, and $\delta \le 1, D = \delta^{-1}$. There exists a constant $C_\eps$ with the following property: Suppose that $P$ is a set of $\delta$-balls and $\T$ is a set of $\delta$-tubes contained in $[0,1]^2$ such that every $p \in P$ intersects at most $K_\beta$ many $\delta$-balls of $P$ (including $p$ itself) and every $t \in \T$ is essentially identical to at most $K_\alpha$ many $\delta$-tubes of $\T$. Let $S = D^{\eps/20}$. Then we have the incidence estimate
\begin{equation}\label{eqn21}
    I(P, \T) \lesim_\eps S^{1/2} \cdot D^{1/2} K_\alpha^{1/2} K_\beta^{1/2} |P|^{1/2} |\T|^{1/2} + S^{-1+\eps/2} I(P^S, \T^S).
\end{equation}
\end{prop}


\begin{proof}
If $K_\alpha = K_\beta = 1$, then the $\delta$-tubes in $\T$ are essentially distinct and the $\delta$-balls in $P$ are pairwise non-intersecting. Thus, we can directly apply \cite[Proposition 2.1]{bradshaw2020incidence} with choice of parameters $\alpha = \frac{\eps^2}{40}$ and weight function $w \equiv 1$.

Now, we will tackle the general case. To do so, we will partition $P$ into $K_\beta$ groups $P_1, P_2, \cdots, P_{K_\beta}$ such that all the balls in $P_i$ are disjoint. Consider a graph on the set of $\delta$-balls of $P$, with two balls connected by an edge if they intersect. Then each ball has maximum degree $K_\beta - 1$ by assumption. To construct the desired partition of $P$, we employ the following well-known lemma from graph theory, which follows from for example Brook's theorem in \cite{brooksthm}:
\begin{lemma}\label{graph}
Any graph with maximum degree $n$ admits a coloring of the vertices with $n+1$ colors such that no two adjacent vertices share the same color.
\end{lemma}
In other words, we may partition of $P$ into $K_\beta$ many sets $P_1, P_2, \dots, P_{K_\beta}$, such that any two intersecting $\delta$-balls in $P$ must belong in different sets of the partition, so the $\delta$-balls in each $P_i$ are disjoint.

Similarly, we may partition $\T$ into $K_\alpha$ groups $\T_1, \T_2, \cdots, \T_{K_\alpha}$ such that the $\delta$-tubes in each $\T_i$ are essentially distinct. Finally, by applying our $K_\alpha = K_\beta = 1$ incidence result to each $P_i$ and $\T_j$, we have
\begin{align*}
    I(P, \T) &\le \sum_{i=1}^{K_\alpha} \sum_{j=1}^{K_\beta} I(P_i, \T_j) \\
        &\lesim_\eps \sum_{i=1}^{K_\alpha} \sum_{j=1}^{K_\beta} \left( S^{1/2} \cdot D^{1/2} |P_i|^{1/2} |\T_j|^{1/2} + S^{-1+\eps/2} \cdot I(P_i^S, \T_j^S) \right) \\
        &= \left( S^{1/2} \cdot D^{1/2} \sum_{i=1}^{K_\alpha} |P_i|^{1/2} \cdot \sum_{j=1}^{K_\beta} |\T_j|^{1/2} \right) + S^{-1+\eps/2} \cdot I(P^S, \T^S) \\
        &\le S^{1/2} \cdot D^{1/2} (K_\alpha^{1/2} |P|^{1/2}) (K_\beta^{1/2}  |\T|^{1/2}) + S^{-1+\eps/2} \cdot I(P^S, \T^S).
\end{align*}
The last line followed from Cauchy-Schwarz and $\sum_{i=1}^{K_\beta} |P_i| = |P|$, $\sum_{j=1}^{K_\alpha} |\T_j| = |\T|$. This proves the desired result \eqref{eqn21} for general $K_\alpha$, $K_\beta \ge 1$.
\end{proof}

\subsection{The upper bound for $I(P, \T)$}
Proposition \ref{twopointone} hints at an inductive approach to upper bound $I(P, \T)$. If the first term in \eqref{eqn21} dominates, we get our desired upper bound. If the second term dominates, then we need to estimate $I(P^S, \T^S)$, where $P^S$ is formed by thickening the $\delta$-balls in $P$ to $S\delta$-balls, and likewise $\T^S$ is formed by thickening the $\delta$-tubes in $\T$ to $S\delta$-tubes. (Here, $S = D^{\eps/20}$.) We thus obtain an incidence problem at scale $S\delta$, so we can apply induction. The key idea is that if $P$ is a $(\delta, \beta, K_\beta)$-set of balls, then $P^S$ is a $(S\delta, \beta, S^\beta K_\beta)$-set, and similarly for tubes $\T^S$. We now prove Theorem \ref{mainbigbig}.

\begin{theorem}\label{bigbig}
Fix $\eps > 0$, and let $c^{-1} = \max(\alpha + \beta - 1, 2)$. There exists a $C_\eps > 0$ such that the following holds: for any $(\delta, \beta, K_\beta)$-set of balls $P$ and $(\delta, \alpha, K_\alpha)$-set of tubes $\T$ contained in $[0, 1]^2$, we have the following incidence bound:
\begin{equation}\label{eqn:bigbig}
    I(P, \T) \le C_\eps \delta^{-c-\eps} (K_\alpha K_\beta)^c |P|^{1-c} |\T|^{1-c}.
\end{equation}
\end{theorem}

\emph{Remark.} If $\alpha = \beta = 2$, which corresponds to the case where there are no constraints on the distribution of $\delta$-tubes or $\delta$-balls in $[0, 1]^2$, the result becomes $I(P, \T) \le C_\eps \delta^{-(1/3 + \eps)} |P|^{2/3} |\T|^{2/3}$, which (up to a $C_\eps \delta^{-\eps}$ factor) recovers a result in \cite{fassler2020planar}.

\begin{proof}
Throughout the proof, we let $D := \delta^{-1}$.

First, we can assume $\eps < \frac{1}{2}$. The proof will be by induction on $n = \lfloor -\log_2 \delta \rfloor$. Let $C_1 (\eps) \ge 1$ be a constant to be chosen later, and $N = N(\eps)$ such that $2C_1 \cdot 2^{-N\eps^2/40} < 1$. Finally, we will choose $C_\eps = \max(2C_1, 2^{3N})$.

The base case will be $\delta \ge 2^{-N}$. Then since $P$ is a $(\delta, \beta, K_\beta)$-set, we have $|P| \le K_\beta D^\beta$. Similarly, $|\T| \le K_\alpha D^\alpha$. Finally, since $\alpha + \beta - 1 \le 3$ and $D \le 2^N$, we get
\begin{equation*}
    I(P, \T) \le |P| |\T| \le D^{c(\alpha + \beta)} (K_\alpha K_\beta)^c |P|^{1-c} |\T|^{1-c} \le 2^{3N} \cdot D^c (K_\alpha K_\beta)^c |P|^{1-c} |\T|^{1-c}
\end{equation*}
This gives the desired bound \eqref{eqn:bigbig} since $2^{3N} \le C_\eps$.

For the inductive step, assume the result is true for $\delta > 2^{-n}$, for some $n \ge N$. We will show the result for $\delta \in (2^{-(n+1)}, 2^{-n}]$.

We first take care of the case when $|P| |\T|$ is small. If $|P| |\T| \le D \cdot K_\alpha K_\beta$, then $I(P, \T) \le |P| |\T| \le D^c (K_\alpha K_\beta)^c |P|^{1-c} |\T|^{1-c}$.

Thus, we may assume $|P| |\T| \ge D \cdot K_\alpha K_\beta$. 
Because $P$ is a $(\delta, \beta, K_\beta)$-set, each $p \in P$ intersects $\le 9K_\beta$ many $\delta$-balls of $P$ (see brief remarks after Definition \ref{def:ball_dim}). Likewise, since $\T$ is a $(\delta, \alpha, K_\alpha)$-set, each $t \in \T$ is essentially identical with $\lesim K_\alpha$ many $\delta$-tubes of $\T$. Thus, we may apply Proposition \ref{twopointone} to obtain, for some constant $C_1 = C_1 (\eps) > 0$ and $S = D^{\eps/20}$,
\begin{equation}\label{twoterms}
    I(P, \T) \le C_1 S^{1/2} \cdot (K_\alpha K_\beta)^{1/2} D^{1/2} |P|^{1/2} |\T|^{1/2} + C_1 S^{-1+\eps/2} I(P^S, \T^S).
\end{equation}
To prove \eqref{eqn:bigbig}, we will show each term is bounded above by $\frac{1}{2} C_\eps D^{c+\eps} (K_\alpha K_\beta)^c |P|^{1-c} |\T|^{1-c}$.

This is clear for the first term, since $c \le \frac{1}{2}, |P| |\T| \ge D \cdot K_\alpha K_\beta$, and $C_\eps \ge 2C_1$.

For the second term, observe that $P^S$ is a $(S\delta, \beta, S^\beta K_\beta)$-set of balls and $\T^S$ is a $(S\delta, \alpha, S^\alpha K_\alpha)$-set of tubes. Thus, by the inductive hypothesis and $c(\alpha+\beta-1) \le 1$, we have
\begin{align*}
    I(P^S, \T^S) &\le C_\eps (S\delta)^{-c - \eps} (K_\alpha K_\beta \cdot S^{\alpha+\beta})^c |P|^{1-c} |\T|^{1-c} \\
            &\le C_\eps S^{1-\eps} \cdot \delta^{-c - \eps} (K_\alpha K_\beta)^c |P|^{1-c} |\T|^{1-c}.
\end{align*}
Recall that $\delta \le 2^{-N}$, and by definition of $N$ and $S$, we get $2C_1 S^{-\eps/2} = 2C_1 D^{-\eps^2/40} < 1$.
Thus, we get $C_1 S^{-1+\eps/2} I(P^S, \T^S) \le \frac{1}{2} C_\eps D^{c+\eps} (K_\alpha K_\beta)^c |P|^{1-c} |\T|^{1-c}$. We have showed that each term of \eqref{twoterms} is bounded above by $\frac{1}{2} C_\eps D^{c+\eps} (K_\alpha K_\beta)^c |P|^{1-c} |\T|^{1-c}$, completing the inductive step and thus the proof of Theorem \ref{bigbig}.

\end{proof}

\section{Proof of Theorems \ref{main}, \ref{thm:furstenberg}, and Corollary \ref{cor:sum-product}}
We restate Theorem \ref{main} here:
\begin{theorem}
    Suppose $\alpha, \beta$ satisfy $0 \le \alpha, \beta \le 2$, and let $K_\alpha, K_\beta \ge 1$. For every $\eps > 0$, there exists $C = C_\eps K_\alpha K_\beta$ with the following property: for every $(\delta, \beta, K_\beta)$-set of balls $P$ and $(\delta, \alpha, K_\alpha)$-set of tubes $\T$, the following bound holds:
    \begin{equation*}
        I(P, \T) \le C \cdot \delta^{-f(\alpha, \beta) - \eps},
    \end{equation*}
    where $f(\alpha, \beta)$ is defined as in Figure \ref{fig:answers}. These bounds are sharp up to $C \cdot \delta^{-\eps}$.
\end{theorem}

\begin{proof}
    The sharpness of these bounds was proved in Section 2, with the constructed examples. We turn to showing the desired upper bounds. In this proof, let $D = \delta^{-1}$.
    
    First, we have $|P| \lesim K_\beta \cdot D^\beta$ and $|\T| \lesim K_\alpha \cdot D^\alpha$ by dimension property (take $w = 2$). We will split into cases.
    
    \begin{itemize}
        \item If $1 \ge \alpha \ge \beta$ or $\alpha \ge 1 \ge \beta \ge \alpha-1$, then we use Theorem \ref{tubesdom} to get
        \begin{equation*}
            I(P, \T)^{a + \beta} \lesim K_\beta^a K_\alpha^\beta D^{a \beta + \eps} D^{\beta^2} D^{a \alpha}.
        \end{equation*}
        
        \item If $1 \ge \beta \ge \alpha$ or $\beta \ge 1 \ge \alpha \ge \beta-1$, then we use Theorem \ref{ballsdom} to get
        \begin{equation*}
            I(P, \T)^{\alpha + b} \lesim K_\beta^\alpha K_\alpha^b D^{\alpha b + \eps} D^{b \beta} D^{\alpha^2}.
        \end{equation*}
        
        \item If $\alpha \ge 1$ and $\beta \ge 1$, then we use Theorem \ref{bigbig} to get
        \begin{equation*}
            I(P, \T) \lesim (K_\alpha K_\beta)^c D^{c+\eps} D^{\alpha(1-c)} D^{\beta(1-c)},
        \end{equation*}
        where $c^{-1} = \max(\alpha + \beta - 1, 2)$.

        \item Suppose $\beta \ge \alpha + 1$. By the short remark after Definition \ref{def:ball_dim}, each $\delta$-ball in $P$ intersects $\le 9K_\beta$ other $\delta$-balls in $P$. Thus, using Lemma \ref{graph}, we can partition $P$ into $P_1, P_2, \cdots, P_{9K_\beta}$ such that the $\delta$-balls in each $P_i$ are disjoint. Using this disjointness property, each $\delta$-tube in $\T$ can only intersect $\lesim D$ many $\delta$-balls of any $P_i$. Thus, we get $I(P_i, \T) \lesim |\T| \cdot D \lesim K_\alpha \cdot D^{\alpha+1}$ for each $1 \le i \le 9K_\beta$, so $I(P, \T) = \sum_{i=1}^{9K_\beta} I(P_i, \T) \lesim K_\alpha K_\beta D^{\alpha+1}$.
        
        \item If $\alpha \ge \beta + 1$, then using a similar partitioning argument as in the previous bullet point, we get $I(P, \T) \lesim K_\alpha K_\beta D^{\beta+1}$.
    \end{itemize}
    
    Combining these results proves Theorem \ref{main}.
\end{proof}

Now we move to the proof of Theorem \ref{thm:furstenberg} and Corollary \ref{cor:sum-product}. We will deduce them from the following incidence estimate:

\begin{theorem}\label{thm:general_furst}
    Fix $\eps > 0$. Suppose $\T$ is an $(\delta, \alpha, K_\alpha)$-set of tubes contained in $[0, 1]^2$. For every $t \in \T$, let $P_t$ be a $(\delta, \beta, K_\beta)$-set of balls contained in $[0, 1]^2$ such that $p \cap t \neq \emptyset$ for each $p \in P_t$. If $c = \max(\alpha + \beta, 2)^{-1}$ and $P = \bigcup_{t \in \T} P_t$, then
    \begin{equation}
        \sum_{t \in \T} |P_t| \lesim_\eps D^{c+\eps} K_\beta^c K_\alpha^c |P|^{1-c} |\T|^{1-c}. \label{eqn:general_furst}
    \end{equation}
\end{theorem}

\textit{Remark.} The LHS of \eqref{eqn:general_furst} is less than $I(P, \T)$: we only count incidences between $t$ and $p \in P_t$, and discard ``stray'' incidences between $t$ and $p \in P \setminus P_t$.

\subsection{Sharpness of Theorem \ref{thm:furstenberg} and idea for Theorem \ref{thm:general_furst}} We first establish the sharpness part of Theorem \ref{thm:furstenberg}. Let $\calC$ be a Cantor set in $[0, 1]$ with Hausdorff dimension $\beta$, and consider the product set $A = \calC \times [0, 1] \subset \R^2$. Then $A$ is a $(\beta, v)$-Furstenberg set for any $0 \le v \le 2$, since lines with angle $\le \frac{1}{100}$ from vertical intersect $A$ in an affine copy of $\calC$, which has dimension $\beta$. Also, $\dim_H (A) = \beta + 1$.

Moving onto the proof of Theorem \ref{thm:general_furst}, it would be nice to assume the set $P$ is a $(\delta, \beta', K_\beta)$-set for some $\beta' < 2$, so that an application of Theorem \ref{bigbig} would be stronger. Unfortunately, a priori $P$ may contain some over-concentrated pockets. To remedy this, we can replace these over-concentrated pockets with a discretized copy of $A$ (from the sharpness of Theorem \ref{thm:furstenberg}). Then we decrease the number of balls in $P$, but increase the number of $\delta$-balls of $P$ intersecting a given tube $t \in \T$. In the end, we obtain a $(\delta, \beta+1, K_\beta)$-set $P'$ with $|P'| \lesim |P|$ but $I(P', \T) \gesim \sum_{t \in \T} |P_t|$, and then we apply Theorem \ref{bigbig} on $P'$ and $\T$ to finish.

\subsection{Proving Theorem \ref{thm:general_furst} with extra assumptions}
It is convenient to make some assumptions about our setup. Fortunately, these extra assumptions are harmless, as we will show in Subsection \ref{subsec:reductions}.
\begin{theorem}\label{thm:general_furst'}
    Fix $\eps > 0$. Suppose $\T$ is an $(\delta, \alpha, K_\alpha)$-set of tubes contained in $[0, 1]^2$. For every $t \in \T$, let $P_t$ be a $(\delta, \beta, K_\beta)$-set of balls contained in $[0, 1]^2$ such that $p \cap t \neq \emptyset$ for each $p \in P_t$. Let $P = \bigcup_{t \in \T} P_t$. Suppose we have the additional simplifying assumptions:
    \begin{enumerate}[(S1)]
        \item\label{S1} $\delta = 2^{-n}$ for some $n \ge 1$, and $K_\alpha, K_\beta \ge 1$ are integers.

        \item\label{S2} All the $\delta$-tubes of $\T$ have angle $[\frac{\pi}{4}, \frac{\pi}{4} + \frac{\pi}{100}]$ with the $y$-axis.

        \item\label{S3} All the $\delta$-balls in $P$ are centered in the lattice $(\delta (2\Z + 1))^2$.
    \end{enumerate}
    If $c = \max(\alpha + \beta, 2)^{-1}$, then
    \begin{equation}
        \sum_{t \in \T} |P_t| \lesim_\eps D^{c+\eps} K_\beta^c K_\alpha^c |P|^{1-c} |\T|^{1-c}. \label{eqn:general_furst'}
    \end{equation}
\end{theorem}

We prove Theorem \ref{thm:general_furst'} in the remainder of this subsection.
As stated in the last subsection, the main idea is to replace $P$ with a $(\delta, \beta+1, K_\beta)$-set $P'$ with $|P'| \lesim |P|$ but $I(P', \T) \gesim \sum_{t \in \T} |P_t|$. A priori, $P$ may contain some over-concentrated pockets, or balls $B_w$ that contain $> K_\beta \cdot (\frac{w}{\delta})^{\beta+1}$ many $\delta$-balls in $P$. We would like to locally replace the portion of $P$ in each over-concentrated pocket with a smaller set of $\delta$-balls (to be constructed later) with cardinality $\le K_\beta \cdot (\frac{w}{\delta})^{\beta+1}$; then the resulting set $P'$ will not have over-concentrated pockets and thus will be a $(\delta, \beta+1, K_\beta)$-set. Unfortunately, this argument does not work because some of the over-concentrated pockets may overlap. Instead, we will find a set of disjoint over-concentrated pockets such that if we fix them, then the new set $P'$ will be a $(\delta, \beta+1, CK_\beta)$-set for some absolute constant $C > 0$. The disjoint over-concentrated pockets will turn out to be dyadic squares, which we define next.
\begin{defn}
Fix $w = 2^{-n}$, $n \ge 0$. The dyadic squares $\calD_w$ are the squares of side length $w$ whose vertices are in the lattice $(w \Z)^2 \cap [0, 1]^2$. 
\end{defn}

We will also adopt the convenient shorthand:

\textbf{Notation.} For a set of $\delta$-balls $P$ and a subset $\cQ \subset \R^2$, let $P \cap \cQ := \{ p \in P \mid p \subset \cQ \}$.

The next well-known lemma roughly says that fixing the over-concentrated dyadic squares is sufficient to ensure $P'$ is a $(\delta, \beta+1, CK_\beta)$-set.

\begin{lemma}\label{lem:easier_dimcheck}
Let $\delta = 2^{-N}$ for some $N \ge 1$. Let $P$ be a set of $\delta$-balls contained in $[0,1]^2$ whose centers lie in $(\delta(2\Z + 1))^2$. Suppose for each $w = 2^{-n}$, $\delta < w < 1$, we have for all $\calQ \in \calD_w$,
\begin{equation}\label{dimcondtemp2}
    |P \cap \cQ| \le K \cdot \left( \frac{w}{\delta} \right)^\beta.
\end{equation}
Then $P$ is a $(\delta, \beta, 64K)$-set of balls.
\end{lemma}

\begin{proof}
Pick a $r$-ball $B_r$ with $r \in [\delta, 1]$; we want to show $|P \cap B_r| \le 64K \cdot (\frac{r}{\delta})^\beta$.

Suppose $r \ge \frac{1}{4}$. Let $D_1, D_2, D_3, D_4$ be the four dyadic squares in $\cD_{1/2}$; their union is $[0, 1]^2$. Furthermore, for each $p \in P \cap [0, 1]^2$, we know that the center of $p$ lies in $(\delta(2\Z + 1))^2$, so $p$ must lie inside some $D_i$. By applying \eqref{dimcondtemp2} to each $D_i$, we have (since $r \ge \frac{1}{4}$ and $\beta \le 2$):
\begin{equation*}
    |P \cap B_r| \le |P| = \sum_{i=1}^4 |P \cap D_i| \le 4K \cdot \left(\frac{1}{2\delta} \right)^\beta \le 64K \cdot \left(\frac{r}{\delta} \right)^\beta.
\end{equation*}
Suppose $\delta \le r < \frac{1}{4}$. Let $w = 2^{-n}$ satisfy $r < w \le 2r$. Let $B_w$ be the $w$-ball concentric with $B_r$, and let $A$ be the point in $(2w \Z)^2$ closest to the center of $B_w$. There are (at most) four dyadic squares $D_1, D_2, D_3, D_4$ in $\calD_{2w}$ with $A$ as a vertex. Using geometric intuition, we see the union $\cup_{i=1}^4 D_i$ contains $B_w \cap [0, 1]^2$, and hence $B_r \cap [0, 1]^2$. Furthermore, since $\frac{w}{\delta} = 2^{N-n}$ is an even integer, a $\delta$-ball $p \in P$ that lies inside $\cup_{i=1}^4 D_i$ must lie inside some $D_i$. By applying \eqref{dimcondtemp2} to each $D_i$ with side length $2w \le 4r < 1$, we have (since $\beta \le 2$):
\begin{equation*}
    |P \cap B_r| \le \sum_{i=1}^4 |P \cap D_i| \le 4K \cdot \left(\frac{2w}{\delta} \right)^\beta \le 4K \cdot \left(\frac{4r}{\delta} \right)^\beta \le 64K \cdot \left( \frac{r}{\delta} \right)^\beta.
\end{equation*}
Hence, $P$ is a $(\delta, \beta, 64K)$-set by Definition \ref{def:ball_dim}.
\end{proof}

Thanks to Lemma \ref{lem:easier_dimcheck}, we shall only look at the set $\cR$ of over-concentrated dyadic squares for $P$, or squares $\cQ \in \cup_{\delta < w < 1 \text{ dyadic}} \cD_w$ satisfying $|P \cap \cQ| \ge K_\beta \cdot (\frac{w}{\delta})^\beta$. The squares in $\cR$ are not necessarily disjoint because a larger dyadic square can contain a smaller dyadic square. However, if we partially order the set of dyadic squares by inclusion, then the set $\cR'$ of maximal elements in $\cR$ with respect to inclusion will be pairwise disjoint. Thus, the final question remaining is: how to fix the over-concentrated pockets in $\cR'$?


Let $w \in (\delta, 1)$ be a dyadic number. We shall construct a set of $\delta$-balls $\calP_w$ contained in $[0, w]^2$ with the following three properties (where the implicit constants are absolute):
\begin{enumerate}[(P1)]

    \item\label{E2} $|\cP_w \cap Q| := |\{ p \in \calP_w \mid p \subset Q \}| \lesim \left( \frac{d}{\delta} \right)^{\beta+1}$ for any $d \times d$ square $Q$ with sides parallel to the coordinate axes and $d \in [\delta, 1]$.

    \item\label{E3} Let $t$ be a $\delta$-tube that forms angle $\frac{\pi}{4} \pm \frac{\pi}{100}$ with the $y$-axis. Suppose $P_t$ is a $(\delta, \beta, K_\beta)$-set satisfying $p \cap t \neq \emptyset$ for all $p \in P_t$, and each $\delta$-ball in $P_t$ is centered in the lattice $(\delta (2\Z + 1))^2$. Then we have $K_\beta |\{ p \in \calP_w \mid p \cap t \neq \emptyset \}| \gesim |P_t \cap [0, w]^2|$.
\end{enumerate}

We defer the construction of $\cP_w$ to Section \ref{subsec:construct Pw}. Now, we will formalize our previous ideas and prove Theorem \ref{thm:general_furst'} assuming the existence of $\calP_w$.

\begin{proof}[Proof of Theorem \ref{thm:general_furst'}]

Recall that for $w = 2^{-n}$, we define $\calD_w$ to be the set of dyadic squares of side length $w$ contained in $[0, 1]^2$. 
Let $\calR$ be the set of squares in $\bigcup_{\delta < w < 1 \text{ dyadic}} \calD_w$ that contain $\ge K_\beta \cdot \left( \frac{w}{\delta} \right)^{\beta+1}$ many $\delta$-balls in $Q$. We call $\calR$ the ``over-concentrated'' squares. Let $\calR'$ be the maximal elements of $\calR$, i.e. the squares $\calQ \in \calR$ such that no square of $\calR$ properly contains $\calQ$. Since two elements of $\calR$ are either disjoint or one lies inside the other, we see that the elements of $\calR'$ are disjoint.

As a notational convenience in this proof, for any subset $A \subset \R^2$, we let $P \cap A := \{ p \in P \mid p \subset A \}$ to be the set of $\delta$-balls in $P$ that lie in $\cQ$. Similarly define $P \setminus A := \{ p \in P \mid p \not\subset A \}$. We will also define the set $\cup \cR' \subset \R^2$ to be the union of the squares in $\cR'$.

Using this notation, we observe an important fact (already noticed in the proof of Lemma \ref{lem:easier_dimcheck}): Since the $\delta$-balls in $P$ are centered in $(\delta(2\Z+1))^2$ (by \ref{S3}), and since the dyadic squares in $\cR'$ have side length being multiples of $2\delta$, we have that any $\delta$-ball in $P$ is either contained in some $\cQ \in \cR'$ or contained in $[0, 1]^2 \setminus \cup \cR'$. This fact will be used throughout the argument without further mention, but let us mention a particular example: we have $P = (P \setminus \cup \cR') \cup \bigsqcup_{\cQ \in \cR'} (P \cap \cQ)$.

We now construct a new set of balls $P'$ that has fewer $\delta$-balls than $P$ in the over-concentrated squares $\cR$ (and equals $P$ outside the set of over-concentrated squares), yet $I(P', \T) \gesim \sum_{p \in P} |P_t|$. For each $\calQ \in \calR'$ with side length $w$, we let $P'(\calQ)$ be a superposition of $K_\beta$ copies of $\calP_w$ placed inside $\calQ$.
Finally, define $P' := (P \setminus \cup \cR') \cup \bigsqcup_{\calQ \in \calR'} P'(\calQ)$, which replaces the $\delta$-balls in $P \cap \calQ$ with $P'(\calQ)$ for each $\calQ \in \cR'$. Then for each $t \in \T$, we have
\begin{align*}
    |\{ p \in P' : p \cap t \neq \emptyset \}| &= |\{ p \in P' \setminus \cup \cR' : p \cap t \neq \emptyset \}| + \sum_{\cQ \in \cR'} |\{ p \in P'(\cQ) \mid p \cap t \neq \emptyset \}| \\
    &\gesim |P_t \setminus \cup \cR'| + \sum_{\cQ \in \cR'} |P_t \cap \cQ| \\
    &= |P_t|.
\end{align*}
(To get from the first to the second line, we used $P \setminus \cup \cR' = P' \setminus \cup \cR'$ and $p \in P_t \implies p \cap t \neq \emptyset$ to lower bound the first term, and property \ref{E3} with our assumptions \ref{S2}, \ref{S3} to lower bound the summation term.) Summing over all $t \in \T$ gives
\begin{equation}\label{eqn:increase incidences 2}
    I(P', \T) := \sum_{t \in \T} |\{ p \in P' : p \cap t \neq \emptyset \}| \gesim \sum_{t \in \T} |P_t|.
\end{equation}
Similarly, by \ref{E2} and the definition of $\calR'$, we have $|P'(\calQ)| \lesim K_\beta \cdot \left( \frac{w}{\delta} \right)^{\beta+1} < |P \cap \cQ|$ for each $\calQ \in \calR'$, so
\begin{equation}\label{eqn:decrease P}
    |P'| = |P' \setminus \cup \cR'| + \sum_{\cQ \in \cR'} |P'(\calQ)| \lesim |P \setminus \cup \cR'| + \sum_{\cQ \in \cR'} |P \cap \cQ| = |P|.
\end{equation}
We now check that $P'$ satisfies the conditions of Lemma \ref{lem:easier_dimcheck} with $\beta+1$ for $\beta$. Pick $\cQ \in \cD_w$ with $w \in (\delta, 1)$ dyadic; if $\cQ \in \cR$, then by definition, $\cQ \subset \cQ'$ for some maximal element $\cQ' \subset \cR'$. Then by estimate \ref{E2} applied to $P'(\cQ')$, we get $|P' \cap \cQ| \lesim K_\beta \left( \frac{w}{\delta} \right)^{\beta+1}$. If $\cQ \notin \cR$, then by definition of $P'$ and $\cR$, we already know $|P' \cap \cQ| = |P \cap \cQ| \le K_\beta \cdot \left( \frac{w}{\delta} \right)^{\beta+1}$. In either case, we get $|P' \cap \cQ| \lesim K_\beta \left( \frac{w}{\delta} \right)^{\beta+1}$.

By assumption, $\T$ is a $(\delta, \alpha, K_\alpha)$-set of tubes, and by Lemma \ref{lem:easier_dimcheck}, $P'$ is a $(\delta, \beta+1, CK_\beta)$-set of balls (here $C > 0$ is an absolute constant). Hence, we can apply Theorem \ref{bigbig}, \eqref{eqn:increase incidences 2}, and \eqref{eqn:decrease P} to get (with $c = \frac{1}{\max(\alpha + \beta, 2)}$):
\begin{equation*}
    \sum_{t \in \T} |P_t| \lesim I(P', \T) \lesim_\eps D^{c+\eps} K_\alpha^c K_\beta^c |P'|^{1-c} |\T|^{1-c} \lesim D^{c+\eps} K_\alpha^c K_\beta^c |P|^{1-c} |\T|^{1-c}.
\end{equation*}
This completes the proof of Theorem \ref{thm:general_furst'}.
\end{proof}

\subsection{Constructing the set $\cP_w$}\label{subsec:construct Pw}
Let $\delta < w \le \frac{1}{2}$ such that $\frac{w}{\delta}$ is an even integer. Recall that we want to construct a set $\cP_w$ contained in $[0, w]^2$ with the following properties (where the implicit constants are absolute):
\begin{enumerate}[(P1)]

    \item $|\cP_w \cap Q| := |\{ p \in \calP_w \mid p \subset Q \}| \lesim \left( \frac{d}{\delta} \right)^{\beta+1}$ for any $d \times d$ square $Q$ with sides parallel to the coordinate axes and $d \in [\delta, 1]$.

    \item Let $t$ be a $\delta$-tube that forms angle $\frac{\pi}{4} \pm \frac{\pi}{100}$ with the $y$-axis. Suppose $P_t$ is a $(\delta, \beta, K_\beta)$-set satisfying $p \cap t \neq \emptyset$ for all $p \in P_t$, and each $\delta$-ball in $P_t$ is centered in the lattice $(\delta (2\Z + 1))^2$. Then we have $K_\beta |\{ p \in \calP_w \mid p \cap t \neq \emptyset \}| \gesim |P_t \cap [0, w]^2|$.
\end{enumerate}

Let $\calC$ be a Cantor set with Hausdorff dimension $\beta$ that contains $0$ and $w$, and let $\calC_\delta := \calC^{(\delta)} \cap \delta \Z$ be a discretization of $\calC$ at scale $\delta$. (Recall that $\calC^{(\delta)}$ is the $\delta$-neighborhood of $\calC$.) Now let $\calP_w$ be the set of $\delta$-balls centered at $(m\delta, n\delta)$, for all $m, n \in \Z$ satisfying $1 \le m, n < \delta^{-1}$ and at least one of $m\delta$ or $n\delta$ belong to $\calC_\delta$.

We first verify \ref{E2}. Let $d \in [\delta, 1]$ and $Q$ be a $d \times d$ square with sides parallel to the coordinate axes, and $I$ be the projection of $Q$ onto the $x$-axis.
Since $\calC$ is a $\beta$-dimensional Cantor set and $I$ has length $d$, we have $|\calC_\delta \cap I| \lesim (\frac{d}{\beta})^\beta$. For any $m\delta \in \calC_\delta \cap I$, there are $\frac{d}{\delta}$ many values of $n$ such that $(m\delta, n\delta) \in Q$. Thus, there are $\lesim (\frac{d}{\delta})^{\beta+1}$ many values for $(m, n) \in \Z^2$ such that $(m\delta, n\delta) \in Q$ and $m\delta \in \calC_\delta$. A similar bound applies to those $(m, n) \in \Z^2$ satisfying $(m\delta, n\delta) \in Q$ and $n\delta \in \calC_\delta$, so we conclude that $|\cP_w \cap Q| \lesim 2 (\frac{d}{\delta})^{\beta+1}$. This proves \ref{E2}.

Now, we show \ref{E3}. We may assume that $t$ intersects $[0, w]^2$ (otherwise the right hand side of \ref{E3} is zero). Since $t$ is a $\delta \times 1$ rectangle and $[0, w]^2$ has diagonal length $\sqrt{2} w < 1$, we see that $t$ must intersect one of the sides of $[0, w]^2$. By rotating the configuration if necessary, we may assume without loss of generality that $t$ intersects the edge $L$ between $(0, 0)$ and $(0, w)$.

If $P_t \cap [0, w]^2 = \emptyset$ we are done; thus, assume there exists $q \in P_t$ with $q \subset [0, w]^2$. In particular, $q \cap t \neq \emptyset$. Let $d$ be the length of the projection of $t \cap [0, w]^2$ onto the $x$-axis. We claim $|\{ p \in \calP_w \mid p \cap t \neq \emptyset \}| \ge \max(1, (\frac{d}{\delta})^\beta)$. To prove the claim, we divide into cases.

\textbf{Case 1.} $d \ge \delta$. Then $t$ intersects some ball $(m\delta, n\delta) \in \calP_w$ for every $m\delta \in \calC_\delta \cap [0, d]$. Since $\calC$ is a $\beta$-dimensional Cantor set containing $0$, we have $|\{ m \in \Z \mid m\delta \in \calC_\delta \cap [0, d]\}| \gesim \left( \frac{d}{\delta} \right)^\beta$. Thus, $|\{ p \in \calP_w \mid p \cap t \neq \emptyset \}| \gesim \left( \frac{d}{\delta} \right)^\beta$.

\textbf{Case 2.} $d < \delta$. We know $q = (m\delta, n\delta)$ for some $1 \le m, n < \frac{w}{\delta}$, $m, n \in 2\Z + 1$. We show that $q$ must have $x$-coordinate $\delta$. Indeed, suppose $q$ has $x$-coordinate $\ge 3\delta$. Then choosing a point $B$ in $q \cap t$, we see that $B$ has $x$-coordinate at least $2\delta$. Since $t$ is convex and also intersects $L$, the projection of $t$ onto the $x$-axis contains $[0, 2\delta]$, so $d > 2\delta$, contradiction. Thus, $q = (\delta, n\delta)$, which means $q \in \calP_w$ (since $\delta \in \calC_\delta)$. Hence, $|\{ p \in \calP_w \mid p \cap t \neq \emptyset \}| \ge 1$.

Thus, we have showed $|\{ p \in \calP_w \mid p \cap t \neq \emptyset \}| \ge \max(1, \frac{d}{\delta})^\beta$.
On the other hand, since the $\delta$-tube $t$ has angle $\frac{\pi}{4} \pm \frac{\pi}{100}$ with the $y$-axis and $t \cap [0, w]^2$ projects onto a length $d$ interval on the $x$-axis, we see that $t \cap [0, w]^2$ is contained in a ball with radius $\sim (d + \delta)$. Thus, $P_t \cap [0, w]^2$ is contained in a ball with radius $\sim (d + 2\delta)$. Since $P_t$ is a $(\delta, \beta, K_\beta)$-set and $\beta \in [0, 2]$, we have
\begin{equation*}
    |P_t \cap [0, w]^2| \lesim K_\beta \left( \frac{d+\delta}{\delta} \right)^\beta
    \lesim K_\beta \max\left(1, \left( \frac{d}{\delta} \right)^\beta \right) \lesim K_\beta |\{ p \in \calP_w \mid p \cap t \neq \emptyset \}|.
\end{equation*}
This verifies \ref{E3}.


\subsection{The simplifying assumptions are harmless.}\label{subsec:reductions}
We will show how to use Theorem \ref{thm:general_furst'} to prove Theorem \ref{thm:general_furst}; it is largely an exercise in pigeonholing.
\begin{proof}[Proof of Theorem \ref{thm:general_furst}]
    First, partition the set of $\delta$-tubes $\T$ into $100$ groups $\T_1, \T_2, \cdots, \T_{100}$, where $\T_i$ consists of the tubes in $\T$ with angle in $[\frac{2\pi i}{100}, \frac{2\pi (i+1)}{100}]$ with the $y$-axis. By the pigeonhole principle, there exists $i$ with $\sum_{t \in \T_i} |P_t| \ge \frac{1}{100} \sum_{t \in \T} |P_t|$. Henceforth, we work only with the tubes in $\T_i$. By rotating the configuration appropriately, we may assume the tubes in $\T_i$ have angle in $[\frac{\pi}{4}, \frac{\pi}{4} + \frac{\pi}{100}]$ with the $y$-axis.

Let $w = 2^{-n}$ satisfy $4\delta \le w < 8\delta$. We first show that every $\delta$-ball in $P$ is contained in a $w$-ball centered at some point in $(w\Z)^2$. Indeed, let $p \in P$, and let $x$ be the point in $(w\Z)^2$ closest to $p$. Then $d(x, p) \le \frac{\sqrt{2}}{2} w < \frac{3}{4}w$ and $\delta + d(x, p) < \delta + \frac{3}{4} w < w$, so $p \subset B_w (x)$.

Now we replace the $\delta$-balls in $P = \bigcup_{t \in \T} P_t$ with $w$-balls centered in $(w\Z)^2$ containing the respective $\delta$-balls, forming $P' = \bigcup_{t \in \T} P_t'$. Likewise, we thicken the $\delta$-tubes in $\T_i$ to $w$-tubes, forming $\T_i'$. The resulting sets $P'_t$ and $\T'_i$ will be $(w, \beta, \lceil 64K_\beta \rceil)$ and $(w, \alpha, \lceil 64K_\alpha \rceil)$-sets respectively.

We would further like $P'$ to have centers in $((2w+1)\Z)^2$. To ensure this, for $a, b \in \{ 0, 1 \}$, let $P'_t (a,b)$ be the elements in $P$ centered at some $(mw, nw)$ with $m \equiv a, n \equiv b \pmod 2$. Then by the pigeonhole principle, there exist $a, b$ such that $\sum_{t \in \T'_i} |P'_t (a,b)| \ge \frac{1}{4} \sum_{t \in \T'_i} |P'_t|$. By translating the configuration appropriately, we may assume that $a = b = 1$.

Let us summarize our achievements: we have found a $(w, \beta, \lceil 64K_\beta \rceil)$-set of tubes $\T'_i$, each forming an angle in $[\frac{\pi}{4}, \frac{\pi}{4} + \frac{\pi}{100}]$ with the $y$-axis, and for each $t \in \T'_i$ we have a $(w, \alpha, \lceil 64K_\alpha \rceil)$ set $P'_t (1,1)$ centered in the lattice $(w (2\Z + 1))^2$, such that $p \cap t \neq \emptyset$ for each $p \in P'_t (1, 1)$ and
\begin{equation}\label{eqn:relation between reduction}
    \sum_{t \in \T} |P_t| \le 400 \sum_{t \in \T_i'} |P'_t (1,1)|.
\end{equation}
Now $\T'_i$ and $P'_t (1,1)$ satisfy the assumptions of Theorem \ref{thm:general_furst'} with parameters $w = 2^{-n}, \lceil 64K_\alpha \rceil, \lceil 64K_\beta \rceil$, so \eqref{eqn:general_furst'} holds for these parameters. Combine this with \eqref{eqn:relation between reduction} to obtain the desired bound \eqref{eqn:general_furst} (with a worse implicit constant).
\end{proof}

\subsection{Proof of Theorem \ref{thm:furstenberg} and Corollary \ref{cor:sum-product}}
We now define a $\delta$-discretized Furstenberg set. Let $D = \delta^{-1}$.
\begin{defn}\label{defn:furstenberg}
For $0 < v \le 2$ and $0 < u \le 1$, we call a collection $P$ of essentially distinct $\delta$-balls a $(\delta, u, v, K_u, K_v)$-Furstenberg set if there exists a $(\delta, v, K_v)$-set of tubes $\T$ with $|\T| \gtrsim K_v^{-1} \cdot D^v$ such that for each $t \in \T$, the set $P_t = \{ p \in P \mid p \cap t \neq \emptyset \}$ is a $(\delta, u, K_u)$-set of balls with $|P_t|\gtrsim K_u^{-1} \cdot D^u.$
\end{defn}
Then by Lemma 3.3 of \cite{hera-improved} with $K_u = K_v = C (\log (\delta^{-1}))^C < C_\epsilon \delta^{-\epsilon}$ for any $\epsilon > 0$, the bound in Theorem \ref{thm:furstenberg} follows from the corresponding discretized version.
\begin{theorem}\label{thm:furstenberg-discrete}
    For $1 \le v \le 2$ and $0 < u \le 1$, a $(\delta, u, v, K_u, K_v)$-Furstenberg set $P$ satisfies $|P| \ge c_\eps K_u^{-3} K_v^{-2} D^{\min(2u + v - 1, u + 1) - \eps},$ for every $\eps>0.$ (Here, $c_\eps > 0$ depends only on $\eps$.)
\end{theorem}

\begin{proof}
By Theorem \ref{thm:general_furst}, we have, for $c = \max(u+v, 2)^{-1}$,
\begin{equation*}
    K_u^{-1} D^u |\T| \le C_\eps D^{c+\eps(1-c)} K_u^c K_v^c |P|^{1-c} |\T|^{1-c}.
\end{equation*}
Using $|\T|\gtrsim K_v^{-1} \cdot D^v$ and $c = \frac{1}{\max(u+v, 2)} \le \frac{1}{2}$, we get
\begin{equation*}
    |P| \ge C_\eps^{-\frac{1}{1-c}} K_u^{-\frac{1+c}{1-c}} K_v^{-\frac{2c}{1-c}} D^{\frac{u+(v-1)c-\eps(1-c)}{1-c}} \ge C_\eps^{-2} K_u^{-3} K_v^{-2} D^{\min(2u+v-1, u+1) - \eps},
\end{equation*}
as desired.
\end{proof}

Finally, we prove Corollary \ref{cor:sum-product}.
\begin{corollary}
    Let $0 < \delta \le 1$, $u, v, v' \in [0, 1]$ with $v + v' > 1$, and $K_u, K_v, K_{v'} \ge 1$. Let $A, B, C \subset [1, 2]$ be sets of disjoint $\delta$-balls such that $A$ is a $(\delta, u, K_u)$-set, $B$ is a $(\delta, v, K_v)$ set, and $C$ is a $(\delta, v', K_{v'})$ set. For a set $E \subset \R$, let $|E|_\delta$ denote the minimum number of $\delta$-balls to cover $E$. Then for $c = \max(u + v + v', 2)^{-1}$,
    \begin{equation*}
        \max(|A+B|_\delta, |A \cdot C|_\delta) \gesim K_u^{-\frac{c}{2(1-c)}} K_v^{-\frac{c}{2(1-c)}} \delta^{\frac{c}{2(1-c)} + \eps} |B|^{\frac{c}{2(1-c)}} |C|^{\frac{c}{2(1-c)}} |A|^{\frac{1}{2(1-c)}}.
    \end{equation*}
\end{corollary}

\begin{proof}
The proof is a very slight modification of Section 6.3 of \cite{furstenberg}, using our Theorem \ref{thm:general_furst}. We provide full technical details below. Let $X$ be a minimal (disjoint) covering of $A + B$ by $\delta$-balls, and let $Y$ be a minimal covering of $AC$ by $\delta$-balls. Let $\tX$ denote the set of centers of the $\delta$-balls in $X$, and define $\tY, \tA, \tB, \tC$ analogously. Finally, for $x \in A + B$, let $\tX(x)$ be the center of the $\delta$-ball in $X$ containing $x$, and similarly for $\tY(x)$. Let
\begin{gather*}
    F = \{ (x, y)^{(\delta)} \mid x \in \tX, y \in \tY \} \text{ is a set of } \delta\text{-balls}, \\
    F_{bc} = \{ (\tX(a+b), \tY(ac))^{(\delta)} \mid a \in \tA \} \subset F, \\
    t_{bc} = \delta\text{-tube with midline } y = cx - bc, \text{ for } b \in \tB, c \in \tC, \\
    \T = \{ t_{bc} \mid b \in \tB, c \in \tC \} \text{ is a set of } \delta\text{-tubes}.
\end{gather*}
(Here, $X^{(\delta)}$ is the $\delta$-neighborhood of $X$.)

We make the following observations.
\begin{enumerate}
    \item $|F| = |\tX| |\tY| = |A+B|_\delta |A\cdot C|_\delta$.
    
    \item Since $B$ is a $(\delta, v, K_v)$-set and $C$ is a $(\delta, v', K_{v'})$ set, $\T$ must be a $(\delta, v+v', 100 K_v K_{v'})$-set of $\delta$-tubes. Furthermore, $|\T| = |B| |C|$.
    
    \item Since $(a+b, ac)$ lies on $\ell_{bc} : y = cx - bc$ and $d((a+b, ac), (\tX(a+b), \tY(ac))) \le \sqrt{2} \delta < \frac{3}{2} \delta$, we see that $t_{bc}$ (the $\frac{\delta}{2}$-neighborhood of $\ell_{bc}$) intersects every $\delta$-ball in $F_{bc}$.
    
    \item We want to show $F_{bc}$ is a $(\delta, u, 4K_u)$-set of balls. Consider a $w$-ball $B_w \subset \R^2$. For each $\delta$-ball $p \in F_{bc}$ that lies in $B_w$, we can find an element $(a+b, ac)$ with $a \in \tA$ such that $p = (\tX(a+b), \tY(ac))^{(\delta)}$. Let $L'$ be the projection of $B_w$ onto the $x$-axis, and $L = \{ x - b \mid x \in L' \}$; then $a \in L$, so $a^{(\delta)} \subset L^{(\delta)}$. We know $|L| \le 2w$, so $|L^{(\delta)}| \le w+2\delta \le 4w$. In other words, $L^{(\delta)}$ is a 1-dimensional ball with radius $\le 2w$. Thus, since $A$ is a $(\delta, u, K_u)$-set, we get
    \begin{equation*}
        \{ p \in F_{bc} \mid p \subset B_w \} \le \{ a^{(\delta)} \in A \mid a^{(\delta)} \subset L^{(\delta)} \} \le K_u \cdot \left( \frac{2w}{\delta} \right)^u.
    \end{equation*}
    This shows $F_{bc}$ is a $(\delta, u, 4K_u)$-set of balls.
\end{enumerate}

Thus, by Theorem \ref{thm:general_furst}, we have
\begin{equation*}
    |\T| |A| \le \sum_{t \in \T} |F_t| \lesim D^{c+\eps(1-c)} K_u^c K_v^c |\T|^{1-c} |F|^{1-c}.
\end{equation*}
Since $|\T| = |B| |C|$, we get $|A+B|_\delta |A\cdot C|_\delta = |F| \gesim K_u^{-\frac{c}{1-c}} K_v^{-\frac{c}{1-c}} D^{-\frac{c}{1-c} - \eps} |B|^{\frac{c}{1-c}} |C|^{\frac{c}{1-c}} |A|^{\frac{1}{1-c}}$, which implies the desired result.
\end{proof}




\section{Appendix}
In this Appendix, we will verify Lemma \ref{app2} for Construction 1, which is the case
\begin{equation}\label{case1}
    \alpha + \beta < 3, \alpha < \beta + 1, \text{ and } \beta < \alpha + 1.
\end{equation}
Recall $D = \delta^{-1}, a = \min(\alpha, 1), b = \min(\beta, 1), \gamma = \frac{a - \alpha + \beta}{a + b}, \kappa = \alpha - (1-\gamma)a = \frac{a \beta + b \alpha - ab}{a + b}$.

\begin{lemma}
(a) We have $0 \le \gamma \le 1$ and $0 \le \kappa \le 1$.

(b) Suppose $\lambda$ is a parameter satisfying the defining condition
\begin{itemize}
    \item $(1-\gamma)a(a+1-\alpha) + \max(\lambda,1-\lambda) \kappa \le a$;
    
    \item $(1-\gamma)b(b+1-\beta) + \max(\lambda,1-\lambda) \kappa \le b$;
    
    \item $\gamma + (1-\gamma)\min(a, b) \ge \max(\lambda, 1-\lambda)$;
    
    \item $0 \le \lambda \le \min(\gamma, 1-\gamma) \le 1$.
\end{itemize}
Then $\T$ is a $(\delta, \alpha)$-set of tubes and $P$ is a $(\delta, \beta)$-set of balls.

(c) $\lambda = \min(\gamma, 1-\gamma)$ satisfies the defining condition.
\end{lemma}

\begin{proof}
We will show four facts:
\begin{enumerate}
    \item $0 \le \gamma \le 1$.
    
    \item $0 \le \kappa \le \min(a, b)$. 
    
    \item The $\delta$-tubes are a $(\delta, \alpha)$-set of tubes.
    
    \item The $\delta$-balls are a $(\delta, \beta)$-set of balls.
\end{enumerate}

These facts allow us to verify $\lambda = \min(\gamma, 1-\gamma)$ satisfies the defining condition.
We perform the following computation using Facts 1-2 and $(a-1)(a-\alpha) = 0$:
\begin{align*}
    (1-\gamma)a(a+1-\alpha) + \gamma\kappa &\le a(1-\gamma) + \gamma a = a,
\end{align*}
\begin{align*}
    (1-\gamma)a(a+1-\alpha) + (1-\gamma)\kappa &= (1-\gamma)(a(a+1-\alpha)+(\alpha - (1 - \gamma) a)) \\
                &= (1-\gamma)(a(a-\alpha) + \alpha + \gamma a) \\
                &= (1-\gamma)(a-\alpha + \alpha + \gamma a) \\
                &= (1-\gamma)(1+\gamma)a \le a.
\end{align*}
\begin{equation*}
    \gamma + (1-\gamma)\min(a,b) \ge \gamma \ge \gamma \kappa,
\end{equation*}
\begin{equation*}
    \gamma + (1-\gamma)\min(a,b) \ge (1-\gamma)\min(a,b) \ge (1-\gamma)\kappa,
\end{equation*}

\begin{figure}[h]
    \centering
    \begin{tabular}{c|c|c|c}
        $\alpha$ & $\beta$ & $\gamma$ & $\kappa$ \\\hline\hline
        $\le1$ & $\le1$ & $\dfrac{\beta}{\alpha + \beta}$ & $\dfrac{\alpha \beta}{\alpha + \beta}$ \\
        $\le1$ & $\ge1$ & $\dfrac{\beta}{\alpha+1}$ & $\dfrac{\alpha \beta}{\alpha + 1}$ \\
        $\ge1$ & $\le1$ & $\dfrac{1-\alpha+\beta}{1+\beta}$ & $\dfrac{\alpha\beta}{1+\beta}$ \\
        $\ge1$ & $\ge1$ & $\dfrac{1-\alpha+\beta}{2}$ & $\dfrac{\alpha+\beta-1}{2}$
    \end{tabular}
    \caption{Computing $\gamma, \kappa$ in terms of $\alpha, \beta$.}
    \label{table}
\end{figure}

Now, we turn to proving the facts. From Figure \ref{table} and the conditions in equation (\ref{case1}), we can easily show Facts 1 and 2. Now, we will verify the $\delta$-tubes are a $(\delta, \alpha)$-set of tubes. Fix $w \in [\delta, 1]$ and a $w \times 2$ rectangle $R_w$; we will count how many $\delta$-tubes are in $R_w$. Recall that the bundles in the construction are arranged in a rectangular grid, with bundles in the same row being translates of each other, and bundles in the same column being rotates of each other. We will estimate the number of bundles per row and column that $R_w$ intersects (where $R_w$ intersects a bundle if it contains a tube from that bundle), as well as the number of tubes $R_w$ can contain from each bundle.

\begin{itemize}
    \item $R_w$ can intersect bundles of $\lesim \lceil \frac{w}{\delta^{\lambda\kappa}} \rceil$ different rows. This is because the angle between bundles of adjacent rows is $\delta^{\lambda \kappa}$, which is at least the angle $\delta^\gamma$ of a single bundle (since $\gamma \ge \lambda \ge \lambda \kappa$). 
    
    \item $R_w$ can intersect bundles of $\lesim \lceil \frac{w}{\delta^{(1-\lambda) \kappa}} \rceil$ different columns. This is because the horizontal spacing between two adjacent columns is $\delta^{(1-\lambda)\kappa}$.
    
    \item For each bundle, $R_w$ can contain $\lesim \min(\lceil \frac{w}{\delta^{\gamma + (1-\gamma)a}} \rceil, D^{(1-\gamma)a})$ $\delta$-tubes. This is because the angle separation between adjacent $\delta$-tubes of the same bundle is $\delta^{\gamma + (1-\gamma)a}$. 
\end{itemize}

Thus, $R_w$ contains exactly $N$ many $\delta$-tubes from $\T$, for
\begin{equation*}
    N \lesim \min\left(\lceil \frac{w}{\delta^{\gamma + (1-\gamma)a}} \rceil, D^{(1-\gamma)a} \right) \cdot \left( \frac{w}{\delta^{(1-\lambda) \kappa}} + 1 \right) \left( \frac{w}{\delta^{\lambda\kappa}} + 1 \right).
\end{equation*}

Suppose $w < \delta^{\gamma + (1-\gamma)a}$. From property of $\lambda$, we get $w < \delta^{\lambda \kappa}$ and $w < \delta^{(1-\lambda)\kappa}$. Hence,
\begin{equation*}
    N \lesim 2 \cdot 2 \cdot 2 = 8 < 8 \left( \frac{w}{\delta} \right)^\alpha.
\end{equation*}
Thus we may assume $w > \delta^{\gamma + (1-\gamma)a}$. In this case, we can use $\lceil x \rceil \le x + 1 \le 2x$ to write
\begin{equation*}
    n \lesim \min\left(\frac{w}{\delta^{\gamma + (1-\gamma)a}}, D^{(1-\gamma)a} \right) \cdot \left( \frac{w}{\delta^{(1-\lambda)\kappa}} + 1 \right) \left( \frac{w}{\delta^{\lambda\kappa}} + 1 \right).
\end{equation*}

Let $m = \min\left(\frac{w}{\delta^{\gamma + (1-\gamma)a}}, D^{(1-\gamma)a} \right)$. We expand the product and bound each term separately. We will use the fact $\min(x, y) \le x^c y^{1-c}$ for any $0 \le c \le 1$, as well as $a \le \min(\alpha, 1)$, $\delta \le w \le 1$, $\kappa = \alpha - (1-\gamma)a$, and the defining relation of $\lambda$.
\begin{equation}\label{just1}
    m \le \left( \frac{w}{\delta^{\gamma + (1-\gamma)a}} \right)^a \left( D^{(1-\gamma)a} \right)^{1-a} = \left( \frac{w}{\delta} \right)^a \le \left( \frac{w}{\delta} \right)^\alpha.
\end{equation}
\begin{equation}\label{just4}
    m \cdot \frac{w}{\delta^{(1-\lambda)\kappa}} \le \left( \frac{w}{\delta} \right)^{\alpha-a} \left( D^{(1-\gamma)a} \right)^{a+1-\alpha} \cdot \frac{w}{\delta^{(1-\lambda)\kappa}} \le \frac{w^{\alpha-a+1}}{\delta^\alpha} \le \left( \frac{w}{\delta} \right)^\alpha,
\end{equation}
\begin{equation}\label{just5}
    m \cdot \frac{w}{\delta^{\lambda\kappa}} \le \left( \frac{w}{\delta} \right)^{\alpha-a} \left( D^{(1-\gamma)a} \right)^{a+1-\alpha} \cdot \frac{w}{\delta^{\lambda\kappa}} \le \frac{w^{\alpha-a+1}}{\delta^\alpha} \le \left( \frac{w}{\delta} \right)^\alpha,
\end{equation}
\begin{equation}\label{just6}
    m \cdot \frac{w}{\delta^{\lambda\kappa}} \cdot \frac{w}{\delta^{(1-\lambda)\kappa}} \le D^{(1-\gamma)a} \cdot \frac{w}{\delta^{\lambda\kappa}} \cdot \frac{w}{\delta^{(1-\lambda)\kappa}} = \frac{w^2}{\delta^\alpha} \le \left( \frac{w}{\delta} \right)^\alpha.
\end{equation}
Hence, we get $|\{ t \in \T : t \subset R_w \}| = N \lesim \left( \frac{w}{\delta} \right)^\alpha$ for all $w \in [\delta, 1]$ and $w \times 2$ rectangles $R_w$. This means $\T$ is a $(\delta, \alpha)$-set of tubes, proving Fact 3. 
    
    

Now, we verify the $\delta$-balls are a $(\delta, \beta)$-set of balls. Fix $w \in [\delta, 1]$ and a ball $B_w$ of radius $w$; we will count how many $\delta$-balls from $P$ are in $B_w$. As before, we will count the number of bundles per row and column that $B_w$ intersects, as well as the number of $\delta$-balls $B_w$ can contain from each bundle.

\begin{itemize}
    \item $B_w$ can intersect bundles of $\lesim \lceil \frac{w}{\delta^{\lambda\kappa}} \rceil$ different rows. This is because the vertical spacing between two adjacent rows is $\delta^{\lambda\kappa}$, which is at least the height $\delta^{1-\gamma}$ of a single bundle (since $1 - \gamma \ge \lambda \ge \lambda \kappa$).
    
    \item $B_w$ can intersect bundles of $\lesim \lceil \frac{w}{\delta^{(1-\lambda)\kappa}} \rceil$ different columns. This is because the horizontal spacing between two adjacent columnns is $\delta^{(1-\lambda)\kappa}$, which is at least the width $\delta^\gamma$ of a single bundle (since $\gamma \ge 1-\lambda \ge (1-\lambda) \kappa$).
    
    \item For each bundle, $B_w$ can contain $\lesim \min(\lceil \frac{w}{\delta^{1-\gamma + \gamma b}} \rceil, D^{\gamma b})$ $\delta$-balls.
\end{itemize}

Thus, $B_w$ contains at most $N$ $\delta$-balls, for
\begin{equation*}
    N \lesim \min\left(\lceil \frac{w}{\delta^{1-\gamma + \gamma b}} \rceil, D^{\gamma b} \right) \cdot \left( \frac{w}{\delta^{(1-\lambda)\kappa}} + 1 \right) \left( \frac{w}{\delta^{\lambda\kappa}} + 1 \right).
\end{equation*}
Using a similar method to the $\delta$-tubes case, we get the desired bound $|\{ p \in P : p \subset B_w \}| = N \lesim \left( \frac{w}{\delta} \right)^\beta$. Thus, $P$ is a $(\delta, \beta)$-set of balls, proving Fact 4. The Lemma is proved. 
    
    
\end{proof}

\bibliographystyle{alpha}
\bibliography{arxiv_version}

\end{document}